\documentclass[reqno,11pt]{article} 
\usepackage{amsmath, amssymb}
\usepackage{amsthm} 
\newcommand{\Mod}[1]{\ (\textup{mod}\ #1)}
\theoremstyle{plain} 
\newtheorem{theorem}{\indent\sc Theorem}[section]
\newtheorem{lemma}[theorem]{\indent\sc Lemma}

\newtheorem{proposition}[theorem]{\indent\sc Proposition}

\theoremstyle{definition} 

\newtheorem{remark}[theorem]{\indent\sc Remark}

\makeatletter
\def\address#1#2{\begingroup
\noindent\parbox[t]{7.8cm}{%
\small{\scshape\ignorespaces#1}\par\vskip1ex
\noindent\small{\itshape E-mail address}%
\/: #2\par\vskip4ex}\hfill%
\endgroup}%
\makeatother
\title{Generation of class fields by Siegel-Ramachandra invariants} 
\author{
\textsc{Dong Hwa Shin} 
}
\date{} 
\begin{document}

\maketitle

\footnote{ 
2010 \textit{Mathematics Subject Classification}. Primary 11G16; Secondary 11F03, 11G15, 11R29,
11R37.}
\footnote{ 
\textit{Key words and phrases}. class field theory, complex multiplication, class
numbers, elliptic and modular units, modular and automorphic
functions.} \footnote{
\thanks{The author was
supported by Hankuk University of Foreign Studies Research Fund of
2014.
} }

\begin{abstract}
We show in many cases that the Siegel-Ramachandra invariants generate
the ray class fields over imaginary
quadratic fields. As its application we revisit the class
number one problem done by Heegner and Stark, and present a new
proof by making use of inequality argument together with
Shimura's reciprocity law.
\end{abstract}

\section{Introduction}

Let $K$ be an imaginary quadratic field with
the ring of integers $\mathcal{O}_K$. For a nontrivial
ideal $\mathfrak{f}$ of $\mathcal{O}_K$, we denote by $\mathrm{Cl}(\mathfrak{f})$
the ray class group modulo $\mathfrak{f}$ and write $C_0$ for its
identity class.  By class field theory there exists a unique
abelian extension $K_\mathfrak{f}$ of $K$, called the \textit{ray class
field} modulo $\mathfrak{f}$, whose Galois group is isomorphic to
$\mathrm{Cl}(\mathfrak{f})$ via the Artin reciprocity map \cite[Chapter
V]{Janusz}. In particular, the ray class field modulo $\mathcal{O}_K$ is
called the \textit{Hilbert class field} of $K$ and is simply written by
$H_K$.
\par
For a rational pair $
\left[\begin{matrix}r_1\\r_2\end{matrix}\right]
\in\mathbb{Q}^2\setminus\mathbb{Z}^2$, the
\textit{Siegel function} $g_{\left[\begin{smallmatrix}r_1\\r_2\end{smallmatrix}\right]}(\tau)$ on the complex upper
half-plane $\mathbb{H} =\{\tau\in\mathbb{C}~|~\mathrm{Im}(\tau)>0\}$
is defined by
\begin{equation}\label{Siegel}
g_{\left[\begin{smallmatrix}r_1\\r_2\end{smallmatrix}\right]}(\tau)=-q^{(1/2)\mathbf{B}_2(r_1)} e^{\pi
ir_2(r_1-1)}(1-q_z)\prod_{n=1}^\infty(1-q^nq_z)(1-q^nq_z^{-1}),
\end{equation}
where $\mathbf{B}_2(X)=X^2-X+1/6$ is the second Bernoulli
polynomial, $q=e^{2\pi i\tau}$ and $q_z=e^{2\pi iz}$ with
$z=r_1\tau+r_2$. It has neither zeros nor poles on $\mathbb{H}$. If
$\mathfrak{f}\neq\mathcal{O}_K$ and $C\in\mathrm{Cl}(\mathfrak{f})$,
then we take any integral ideal $\mathfrak{c}$ in $C$ and $z_1,z_2\in\mathbb{C}$ such that
$\mathfrak{f}\mathfrak{c}^{-1}=\mathbb{Z}z_1+\mathbb{Z}z_2$ and
$z=z_1/z_2\in\mathbb{H}$. We then define the \textit{Siegel-Ramachandra
invariant} modulo $\mathfrak{f}$ at $C$ by
\begin{equation}\label{S-R}
g_\mathfrak{f}(C)=g_{\left[\begin{smallmatrix}a/N\\b/N\end{smallmatrix}\right]}(z)^{12N},
\end{equation}
where $N$ is the smallest positive integer in $\mathfrak{f}$ and
$a$, $b$ are integers such that $1=(a/N)z_1+(b/N)z_2$.
 This value depends only on the class $C$
\cite[Chapter 2, Remark to Theorem 1.2]{K-L}, and lies in
$K_\mathfrak{f}$ \cite[Chapter 2, Proposition 1.3 and Chapter 11,
Theorem 1.1]{K-L}. Furthermore, it satisfies the transformation
formula
\begin{equation}\label{Artin}
g_\mathfrak{f}(C_1)^{\sigma(C_2)}=g_\mathfrak{f}(C_1C_2)\quad(C_1,C_2\in\mathrm{Cl}(\mathfrak{f})),
\end{equation}
where $\sigma$ is the Artin reciprocity map \cite[pp.235--236]{K-L}.
\par
In 1964 Ramachandra \cite[Theorem 10]{Ramachandra} first
constructed a primitive generator of $K_\mathfrak{f}$ over $K$ for
any $\mathfrak{f}\neq\mathcal{O}_K$, however, his invariant involves
overly complicated product of Siegel-Ramachandra invariants and the
singular values of the modular $\Delta$-function. Thus, Lang
\cite[p.292]{Lang} and Schertz \cite[p.386]{Schertz} conjectured
that the simplest invariant $g_\mathfrak{f}(C_0)$ would be a
primitive generator of $K_\mathfrak{f}$ over $K$ (or, over $H_K$),
and Schertz gave a conditional proof \cite[Theorems 3 and
4]{Schertz}.
\par
In this paper we shall first show in $\S$\ref{sec3} that when
$\mathfrak{f}=(N)$ for an integer $N$ ($\geq2$), $g_\mathfrak{f}(C_0)$
generates $K_{(N)}$ over $H_K$ for almost all imaginary quadratic
fields $K$ (Theorem \ref{overH}). We shall further develop a
simple criterion for $g_\mathfrak{f}(C_0)$ to be a primitive
generator of $K_\mathfrak{f}$ over $K$ when $\mathfrak{f}$ is just a
nontrivial ideal of $\mathcal{O}_K$
(Theorem \ref{overK} and Remark
\ref{degree}) by adopting Schertz's idea. In $\S$\ref{sec4} we shall
investigate some properties of Siegel-Ramachandra invariants modulo $2$.
\par
Gauss' class number one problem for imaginary quadratic fields was
first solved by Heegner \cite{Heegner} in 1952. There was a
gap in his proof which heavily relies on the singular values of the
Weber functions, however, few years later complete proofs were found
independently by Baker \cite{Baker} and Stark \cite{Stark}.
Moreover, Stark \cite{Stark} finally filled up the supposed gap in
Heegner's proof. In $\S$\ref{sec5} as an application we shall
introduce a new proof (Theorems \ref{generator} and
\ref{classnumberone}) by using Siegel functions and Stevenhagen's
explicit description of Shimura's reciprocity law \cite[$\S$3,
6]{Stevenhagen}.

\section{Preliminaries}

First, we shall briefly review necessary basic properties of Siegel
functions and Shimura's reciprocity law.
\par
For a positive integer $N$ let $\zeta_N=e^{2\pi i/N}$ be a primitive $N$-th root of unity and
\begin{eqnarray*}
\Gamma(N)=\{\gamma\in
\mathrm{SL}_2(\mathbb{Z})~|~
\gamma\equiv
I_2\Mod{N}\}
\end{eqnarray*}
be the principal congruence subgroup of level $N$ of
$\mathrm{SL}_2(\mathbb{Z})$. Then its corresponding modular curve of
level $N$ is denoted by
$X(N)=\Gamma(N)\backslash(\mathbb{H}\cup\mathbb{P}^1(\mathbb{Q}))$.
Furthermore, we let $\mathcal{F}_N$ be the field of meromorphic
functions on $X(N)$ defined over the $N$-th cyclotomic
field $\mathbb{Q}(\zeta_N)$. We know that
$\mathcal{F}_1=\mathbb{Q}(j(\tau))$, where
\begin{equation}\label{jFourier}
j(\tau)=q^{-1}+744+196884q+21493760q^2+864299970q^3+20245856256q^4+\cdots
\end{equation}
is the elliptic modular $j$-function, and $\mathcal{F}_N$ is a
Galois extension of $\mathcal{F}_1$ with
\begin{equation}\label{Gal(F_N/F_1)}
\mathrm{Gal}(\mathcal{F}_N/\mathcal{F}_1)\simeq\mathrm{GL}_2(\mathbb{Z}/N\mathbb{Z})/\{\pm I_2\},
\end{equation}
whose action is given as follows: For an element
$\alpha\in\mathrm{GL}_2(\mathbb{Z}/N\mathbb{Z})/\{\pm I_2\}$ we
decompose it into
\begin{equation*}
\alpha=\alpha_1\cdot\alpha_2~\textrm{for some}~
\alpha_1\in\mathrm{SL}_2(\mathbb{Z})~\textrm{and}~
\alpha_2=\left[\begin{matrix}1&0\\0&d\end{matrix}\right]
~\textrm{with}~d\in(\mathbb{Z}/N\mathbb{Z})^*.
\end{equation*}
Then, the action of $\alpha_1$ is given by a fractional linear
transformation. And, $\alpha_2$ acts by the rule
\begin{equation*}
\sum_{n\gg-\infty} c_nq^{n/N}\mapsto \sum_{n\gg-\infty}
c_n^{\sigma_d}q^{n/N},
\end{equation*}
where $\sum_{n\gg-\infty} c_nq^{n/N}$ is the Fourier expansion of a
function in $\mathcal{F}_N$ and $\sigma_d$ is the automorphism of
$\mathbb{Q}(\zeta_N)$ defined by $\zeta_N^{\sigma_d}=\zeta_N^d$
\cite[Chapter 6, $\S$3]{Lang}. Here, for later use, we observe that
\begin{equation}\label{size}
[\mathcal{F}_N:\mathcal{F}_1]=\#\mathrm{GL}_2(\mathbb{Z}/N\mathbb{Z})/\{\pm I_2\}=
\left\{\begin{array}{ll} 6&
\textrm{if}~N=2,\vspace{0.1cm}\\
(N^4/2)\prod_{p|N}(1-p^{-1})(1-p^{-2})& \textrm{if}~N\geq3
\end{array}\right.
\end{equation}
\cite[pp.21--22]{Shimura}.

\begin{proposition}\label{level}
For a given integer $N$ \textup{(}$\geq2$\textup{)} let
$\{m(\mathbf{r})\}_{
\mathbf{r}\in(1/N)\mathbb{Z}^2\setminus\mathbb{Z}^2}$ be a
family of integers such that $m(\mathbf{r})=0$ except finitely many
$\mathbf{r}$. A product of Siegel functions
\begin{equation*}
g(\tau)=\zeta
\prod_{\mathbf{r}=\left[\begin{smallmatrix}r_1\\r_2\end{smallmatrix}\right]}
g_{\mathbf{r}}(\tau)^{m(\mathbf{r})}
\end{equation*}
belongs to $\mathcal{F}_N$,
where $\zeta=\prod_{\mathbf{r}}e^{\pi ir_2(1-r_1)m(\mathbf{r})}$,
if
\begin{eqnarray*}
&&\sum_{\mathbf{r}} m(\mathbf{r})(Nr_1)^2\equiv\sum_{\mathbf{r}}
m(\mathbf{r})(Nr_2)^2\equiv0\Mod{\gcd(2,N)\cdot N},\\
&&\sum_{\mathbf{r}} m(\mathbf{r})(Nr_1)(Nr_2)\equiv0\Mod{N},\\
&&\gcd(12, N)\cdot\sum_{\mathbf{r}} m(\mathbf{r})\equiv0\Mod{12}.
\end{eqnarray*}
\end{proposition}
\begin{proof}
See \cite[Chapter 3, Theorems 5.2 and 5.3]{K-L}.
\end{proof}

\begin{remark}
Let $g(\tau)$ be an element of $\mathcal{F}_N$ for some integer $N$
($\geq2$). If both $g(\tau)$ and $g(\tau)^{-1}$ are integral over
$\mathbb{Q}[j(\tau)]$, then $g(\tau)$ is called a \textit{modular unit} (of
level $N$) . As is well-known, $g(\tau)$ is a modular unit if and
only if it has neither zeros nor poles on $\mathbb{H}$ (\cite[p.36]{K-L}
or \cite[Theorem 2.2]{K-S}). Hence any product of Siegel functions
becomes a modular unit. In particular,
$g_{\left[\begin{smallmatrix}r_1\\r_2\end{smallmatrix}\right]}(\tau)^{12N/\gcd(6,N)}$ is a modular unit of level $N$
for any $\left[\begin{matrix}r_1\\r_2\end{matrix}\right]\in(1/N)\mathbb{Z}^2\setminus\mathbb{Z}^2$.
\end{remark}

For a real number $x$ we denote by $\langle x\rangle$ the fractional
part of $x$ in the interval $[0,1)$.

\begin{proposition}\label{transform}
Let $\left[\begin{matrix}r_1\\r_2\end{matrix}\right]\in(1/N)\mathbb{Z}^2\setminus\mathbb{Z}^2$ for an integer $N$
\textup{(}$\geq2$\textup{)}.
\begin{itemize}
\item[\textup{(i)}] We have the $q$-order formula
\begin{equation*}
\mathrm{ord}_q~g_{\left[\begin{smallmatrix}r_1\\r_2\end{smallmatrix}\right]}(\tau)=\frac{1}{2}\mathbf{B}_2 (\langle
r_1\rangle).
\end{equation*}
\item[\textup{(ii)}]
For
$\gamma=\left[\begin{matrix}a&b\\c&d\end{matrix}\right]\in\mathrm{SL}_2(\mathbb{Z})$
with $c>0$ we get the transformation formula
\begin{equation*}
g_{\left[\begin{smallmatrix}r_1\\r_2\end{smallmatrix}\right]}(\tau)\circ\gamma=-ie^{(\pi
i/6)(a/c+d/c-12\sum_{k=1}^{c-1} (k/c-1/2)(\langle kd/c\rangle-1/2))}
g_{
\left[\begin{smallmatrix}r_1a+r_2c\\r_1b+r_2d\end{smallmatrix}\right]
}(\tau).
\end{equation*}
\item[\textup{(iii)}] For $s=\left[\begin{matrix}s_1\\s_2\end{matrix}\right]\in\mathbb{Z}^2$ we have
\begin{equation*}
g_{\left[\begin{smallmatrix}
r_1+s_1\\r_2+s_2\end{smallmatrix}\right]}(\tau) =(-1)^{s_1s_2+s_1+s_2}e^{-\pi
i(s_1r_2-s_2r_1)}g_{\left[\begin{smallmatrix}r_1\\r_2\end{smallmatrix}\right]}(\tau).
\end{equation*}
\item[\textup{(iv)}]
$g_{\left[\begin{smallmatrix}r_1\\r_2\end{smallmatrix}\right]}(\tau)^{12N/\gcd(6,N)}$ is determined only by
$\pm\left[\begin{matrix}r_1\\r_2\end{matrix}\right]\Mod{\mathbb{Z}^2}$.
\item[\textup{(v)}] An element
$\left[\begin{matrix}a&b\\c&d\end{matrix}\right]\in\mathrm{GL}_2(\mathbb{Z}/N\mathbb{Z})/\{\pm I_2\}$
\textup{(}$\simeq
\mathrm{Gal}(\mathcal{F}_N/\mathcal{F}_1)$\textup{)} acts on it by
\begin{equation*}
(g_{\left[\begin{smallmatrix}r_1\\r_2\end{smallmatrix}\right]}(\tau)^{12N/\gcd(6,N)})
^{\left[\begin{smallmatrix}a&b\\c&d\end{smallmatrix}\right]}
=g_{\left[\begin{smallmatrix}r_1a+r_2c\\
r_1b+r_2d\end{smallmatrix}\right]}(\tau)^{12N/\gcd(6,N)}.
\end{equation*}
\item[\textup{(vi)}] $g_{\left[\begin{smallmatrix}r_1\\r_2\end{smallmatrix}\right]}(\tau)$ is integral over
$\mathbb{Z}[j(\tau)]$.
\end{itemize}
\end{proposition}
\begin{proof}
(i) See \cite[p.31]{K-L}.\\
(ii) See \cite[p.27, K1 and p.29]{K-L} and \cite[Chapter IX]{Lang3}.\\
(iii) See \cite[p.28, K2 and p.29]{K-L}.\\
(iv) One can easily check this relation by the definition (\ref{Siegel}) and (iii).\\
(v) See \cite[Chapter 2, Proposition 1.3]{K-L} and (iv).\\
(vi) See \cite[$\S$3]{K-S}.
\end{proof}

For an imaginary quadratic field $K$ of discriminant $d_K$ we let
\begin{equation}\label{theta}
\tau_K=\left\{\begin{array}{ll}
\sqrt{d_K}/2&\textrm{if}~d_K\equiv0\Mod{4},
\vspace{0.1cm}\\
(3+\sqrt{d_K})/2&\textrm{if}~ d_K\equiv1\Mod{4},\end{array}\right.
\end{equation}
which generates the ring of integers $\mathcal{O}_K$ of $K$ over
$\mathbb{Z}$. Then we have
\begin{equation*}
\min(\tau_K,\mathbb{Q})=X^2+BX+C=\left\{
\begin{array}{ll}
X^2-d_K/4 & \textrm{if}~d_K\equiv0\Mod{4},
\vspace{0.1cm}\\
X^2-3X+(9-d_K)/4 & \textrm{if}~d_K\equiv1\Mod{4}.
\end{array}\right.
\end{equation*}
For each positive integer $N$ we define the matrix group
\begin{equation*}
W_{N,\tau_K}=\bigg\{\left[\begin{matrix}t-Bs &
-Cs\\s&t\end{matrix}\right]\in\mathrm{GL}_2(\mathbb{Z}/N\mathbb{Z})~|~t,s\in\mathbb{Z}/N\mathbb{Z}\bigg\}.
\end{equation*}

\begin{proposition}\label{CM} For a positive integer $N$ we have
\begin{equation*}
K_{(N)}=K(h(\tau_K)~|~h\in\mathcal{F}_N~\textrm{is
defined and finite at $\tau_K$}).
\end{equation*}
\end{proposition}
\begin{proof}
See \cite[Chapter 10, Corollary to Theorem 2]{Lang}  or
\cite[Proposition 6.33]{Shimura}.
\end{proof}

\begin{proposition}[Shimura's reciprocity law]\label{Gee}
Let $K$ be an imaginary quadratic field. For each positive integer
$N$, the matrix group $W_{N,\tau_K}$ gives rise to the surjection
\begin{eqnarray*}
W_{N,\tau_K}&\longrightarrow&\mathrm{Gal}(K_{(N)}/H_K)\\
\alpha&\mapsto&(h(\tau_K)\mapsto h^\alpha(\tau_K)~|~
\textrm{$h(\tau)\in\mathcal{F}_N$ is defined and finite at
$\tau_K$}),
\end{eqnarray*}
whose kernel is
\begin{equation*}
\mathrm{Ker}_{N,\tau_K}=\left\{\begin{array}{ll}
\bigg\{\pm\left[\begin{matrix}1&0\\0&1\end{matrix}\right],~
\pm\left[\begin{matrix}-1&-3\\1&-2\end{matrix}\right],~
\pm\left[\begin{matrix}-2&3\\-1&1\end{matrix}\right] \bigg\} &
\textrm{if}~K=\mathbb{Q}(\sqrt{-3}),
\vspace{0.1cm}\\
\bigg\{\pm\left[\begin{matrix}1&0\\0&1\end{matrix}\right],~
\pm\left[\begin{matrix}0&-1\\1&0\end{matrix}\right] \bigg\} &
\textrm{if}~K=\mathbb{Q}(\sqrt{-1}),
\vspace{0.1cm}\\
\bigg\{\pm\left[\begin{matrix}1&0\\0&1\end{matrix}\right]\bigg\} &
\textrm{otherwise.}
\end{array}\right.
\end{equation*}
\end{proposition}
\begin{proof}
See \cite[$\S$3]{Stevenhagen} or \cite[pp.50--51]{Gee}.
\end{proof}

For an imaginary quadratic field $K$ of discriminant $d_K$, let
\begin{align*}
\mathrm{C}(d_K)=\{aX^2+bXY+cY^2\in\mathbb{Z}[X,Y]~|~
\gcd(a,b,c)=1,~b^2-4ac=d_K,\\(-a<b\leq a<c~\textrm{or}~0\leq b\leq
a=c)\}
\end{align*}
be the form class group of reduced quadratic forms of discriminant
$d_K$, whose identity element is
\begin{equation*}
\left\{\begin{array}{ll} X^2-(d_K/4)Y^2 &
\textrm{if}~d_K\equiv0\Mod{4},
\vspace{0.1cm}\\
 X^2+XY+((1-d_K)/4)Y^2 &
\textrm{if}~d_K\equiv1\Mod{4}
\end{array}\right.
\end{equation*}
\cite[Theorems 2.8 and 3.9]{Cox}. Note that if
$aX^2+bXY+cY^2\in\mathrm{C}(d_K)$, then
\begin{equation*}
a\leq\sqrt{|d_K|/3}
\end{equation*}
\cite[p.29]{Cox} and the group $\mathrm{C}(d_K)$ is isomorphic to
the ideal class group of $K$, and hence to $\mathrm{Gal}(H_K/K)$
\cite[Theorem 7.7]{Cox}. Thus, in particular, the class number of
$K$ is the same as the order of the group $\mathrm{C}(d_K)$, namely
$[H_K:K]$. We denote it by $h_K$.

\begin{proposition}[Shimura's reciprocity law]\label{Hilbert}
Let $K$ be an imaginary quadratic field of discriminant $d_K$, and
$p$ be a prime. For each $Q=aX^2+bXY+cY^2\in\mathrm{C}(d_K)$ let
\begin{equation*}
\tau_Q=(-b+\sqrt{d_K})/2a\quad(\in\mathbb{H})
\end{equation*}
and $u_Q$ be an element of
$\mathrm{GL}_2(\mathbb{Z}/p\mathbb{Z})/\{\pm I_2\}$ given as follows:
\begin{itemize}
\item[] \textup{Case 1.} $d_K\equiv0\Mod{4}$
\begin{equation*}
u_Q=\left\{\begin{array}{ll}
\left[\begin{matrix}a&b/2\\0&1\end{matrix}\right]&\textrm{if}~p\nmid a,\vspace{0.1cm}\\
\left[\begin{matrix}-b/2 &-c\\1&0\end{matrix}\right]&\textrm{if}~p\mid a~\textrm{and}~p\nmid c,\vspace{0.1cm}\\
\left[\begin{matrix}-a-b/2&-c-b/2\\1&-1\end{matrix}\right]&\textrm{if}~p\mid a~\textrm{and}~p\mid
c,
\end{array}\right.
\end{equation*}
\item[] \textup{Case 2.} $d_K\equiv1\Mod{4}$
\begin{equation*}
u_Q=\left\{\begin{array}{ll}
\left[\begin{matrix}a& (3+b)/2\\0&1\end{matrix}\right]&\textrm{if}~p\nmid a,\vspace{0.1cm}\\
\left[\begin{matrix}(3-b)/2&-c\\1&0\end{matrix}\right]&\textrm{if}~p\mid a~\textrm{and}~p\nmid c,\vspace{0.1cm}\\
\left[\begin{matrix}-a+(3-b)/2&-c-(3+b)/2\\1&-1\end{matrix}\right]&\textrm{if}~p\mid a~\textrm{and}~p\mid
c.
\end{array}\right.
\end{equation*}
\end{itemize}
If $h(\tau)\in\mathcal{F}_p$ is defined and finite at $\tau_K$ and
$h(\tau_K)\in H_K$, then the conjugates of $h(\tau_K)$ via the
action of $\mathrm{Gal}(H_K/K)$ are given by
\begin{equation*}
h^{u_Q}(\tau_Q) \quad(Q\in\mathrm{C}(d_K))
\end{equation*}
possibly with some multiplicity.
\end{proposition}
\begin{proof}
See \cite[$\S$6]{Stevenhagen} or \cite[Lemma 20]{Gee}.
\end{proof}

\section{Generators of ray class fields}\label{sec3}

Let $K$ be an imaginary quadratic field. For an integer $N$
($\geq2$) we get
\begin{equation*}
g_{(N)}(C_0)=g_{
\left[\begin{smallmatrix}0\\1/N\end{smallmatrix}\right]}(\tau_K)^{12N}
\end{equation*}
by the definition (\ref{S-R}). In this section we shall show that it
plays a role of primitive generator of $K_{(N)}$ over $H_K$ (or,
even over $K$).

\begin{lemma}\label{function}
Let $\left[\begin{matrix}s\\t\end{matrix}\right]\in\mathbb{Z}^2\setminus N\mathbb{Z}^2$  for an integer $N$
\textup{(}$\geq2$\textup{)}. If $\left[\begin{matrix}s\\t\end{matrix}\right]\not\equiv\pm
\left[\begin{matrix}0\\1\end{matrix}\right]\Mod{N}$,
then $g_{\left[\begin{smallmatrix}0\\1/N\end{smallmatrix}\right]}(\tau)^{12N}\neq g_{\left[\begin{smallmatrix}s/N\\t/N\end{smallmatrix}\right]}(\tau)^{12N}$.
\end{lemma}
\begin{proof}
Assume on the contrary that $g_{\left[\begin{smallmatrix}0\\1/N\end{smallmatrix}\right]}(\tau)^{12N}=
g_{\left[\begin{smallmatrix}s/N\\t/N\end{smallmatrix}\right]}(\tau)^{12N}$. Since
\begin{equation*}
\mathrm{ord}_q~
g_{\left[\begin{smallmatrix}0\\1/N\end{smallmatrix}\right]}(\tau)^{12N}
=6N\mathbf{B}_2(0)=\mathrm{ord}_q~
g_{\left[\begin{smallmatrix}s/N\\t/N\end{smallmatrix}\right]}(\tau)^{12N}=6N\mathbf{B}_2(\langle s/N\rangle)
\end{equation*}
by Proposition \ref{transform} (i), we must have $s\equiv0\Mod{N}$
by the graph of $\mathbf{B}_2(X)=X^2-X+1/6$. And, since
\begin{eqnarray*}
&&\mathrm{ord}_q~(g_{\left[\begin{smallmatrix}0\\1/N\end{smallmatrix}\right]}(\tau)^{12N})^{\left[
\begin{smallmatrix}0&-1\\1&0\end{smallmatrix}\right]}=
\mathrm{ord}_q~g_{\left[\begin{smallmatrix}1/N\\0\end{smallmatrix}\right]}(\tau)^{12N}=
6N\mathbf{B}_2(1/N)\\
&=&\mathrm{ord}_q~(g_{\left[\begin{smallmatrix}0\\t/N\end{smallmatrix}\right]}(\tau)^{12N})^{\left[
\begin{smallmatrix}0&-1\\1&0\end{smallmatrix}\right]}=
\mathrm{ord}_q~g_{\left[\begin{smallmatrix}t/N\\0\end{smallmatrix}\right]}(\tau)^{12N}= 6N\mathbf{B}_2(\langle
t/N\rangle)
\end{eqnarray*}
by Proposition \ref{transform} (ii) and (i), it follows that
$t\equiv\pm1\Mod{N}$. This proves the lemma.
\end{proof}

\begin{lemma}\label{j}
\begin{itemize}
\item[\textup{(i)}] $j(\tau)$ induces a bijective map $j:\mathrm{SL}_2(\mathbb{Z})\backslash\mathbb{H}
\rightarrow\mathbb{C}$.
\item[\textup{(ii)}] If $K_1$ and $K_2$ are distinct imaginary quadratic
fields, then $\tau_{K_1}$ and $\tau_{K_2}$ are not equivalent
under the action of $\mathrm{SL}_2(\mathbb{Z})$.
\end{itemize}
\end{lemma}
\begin{proof}
(i) See \cite[Chapter 3, Theorem 4]{Lang}.\\
(ii) See \cite[Chapter 3, Theorem 1]{Lang}.
\end{proof}

For a real number $x$ we denote by $[x]$ the
greatest integer that is less than or equal to $x$.

\begin{theorem}\label{overH}
For a given integer $N$ \textup{(}$\geq2$\textup{)} we have
\begin{eqnarray*}
&&\#\{\textrm{imaginary quadratic fields
$K$}~|~\textrm{$g_{\left[\begin{smallmatrix}0\\1/N\end{smallmatrix}\right]}(\tau_K)^{12N}$ does not generate
$K_{(N)}$ over $H_K$}\}\\
&\leq&\left\{\begin{array}{ll}
12 & \textrm{if}~N=2,\vspace{0.1cm}\\
((N+1)[N/2]-1)(N^5/4)\prod_{p|N}(1-p^{-1})(1-p^{-2}) &
\textrm{if}~N\geq3.
\end{array}\right.
\end{eqnarray*}
\end{theorem}
\begin{proof}
Let
\begin{equation*}
S=\bigg\{\left[\begin{matrix}s\\t\end{matrix}\right]\in\mathbb{Z}^2~|~(s=0,~2\leq t\leq[N/2])~\textrm{or}~
(1\leq s\leq[N/2],~0\leq t\leq N-1)\bigg\},
\end{equation*}
which consists of $(N+1)[N/2]-1$ elements. For each $\left[\begin{matrix}s\\t
\end{matrix}\right]\in S$ we
consider the function
\begin{equation*}
g(\tau)=g_{\left[\begin{smallmatrix}0\\1/N\end{smallmatrix}\right]}(\tau)^{12N}-
g_{\left[\begin{smallmatrix}s/N\\t/N\end{smallmatrix}\right]}(\tau)^{12N}
\quad(\in\mathcal{F}_N),
\end{equation*}
which is nonzero by Lemma \ref{function}. Since $g(\tau)$ is
integral over $\mathbb{Z}[j(\tau)]$ by Proposition
\ref{transform} (vi), we have
\begin{equation}\label{poly}
\mathbf{N}_{\mathcal{F}_N/\mathcal{F}_1}(g(\tau))=g(\tau)\prod_{\sigma\neq\mathrm{id}}
g(\tau)^\sigma=P(j(\tau)) \quad \textrm{for some nonzero polynomial
$P(X)\in\mathbb{Z}[X]$.}
\end{equation}
Note by Proposition \ref{transform} (v) that any conjugate of
$g(\tau)$ under the action of
$\mathrm{Gal}(\mathcal{F}_N/\mathcal{F}_1)$ is of the form
\begin{equation*}
g_{\left[\begin{smallmatrix}a/N\\b/N\end{smallmatrix}\right]}(\tau)^{12N}-
g_{\left[\begin{smallmatrix}c/N\\d/N\end{smallmatrix}\right]}(\tau)^{12N} \quad\textrm{for
some}~\left[\begin{matrix}a\\b\end{matrix}\right],
\left[\begin{matrix}c\\d\end{matrix}\right]\in\mathbb{Z}^2-N\mathbb{Z}^2,
\end{equation*}
which is holomorphic on $\mathbb{H}$. Now, let
\begin{equation*}
Z_{\left[\begin{smallmatrix}s\\t\end{smallmatrix}\right]}=\{\textrm{imaginary quadratic fields}~K~|~g(\tau_K)=0\}.
\end{equation*}
If $K\in Z_{\left[\begin{smallmatrix}s\\t\end{smallmatrix}\right]}$, then (\ref{poly}) gives $P(j(\tau_K))=0$,
from which we obtain by Lemma \ref{j}
\begin{equation*}
\#Z_{\left[\begin{smallmatrix}s\\t\end{smallmatrix}\right]}\leq\deg P(X).
\end{equation*}
On the other hand, since
\begin{eqnarray*}
&&\mathrm{ord}_q~( g_{\left[\begin{smallmatrix}a/N\\b/N\end{smallmatrix}\right]}(\tau)^{12N}-
g_{\left[\begin{smallmatrix}c/N\\d/N\end{smallmatrix}\right]}(\tau)^{12N})\\
&\geq& \min\big\{6N\mathbf{B}_2(\langle a/N\rangle),6N
\mathbf{B}_2(\langle c/N\rangle)\big\}\quad\textrm{by Proposition
\ref{transform} (i)}\\
&\geq&6N\mathbf{B}_2(1/2) \quad\textrm{by
the graph of $\mathbf{B}_2(X)=X^2-X+1/6$}\\
&=&-N/2,
\end{eqnarray*}
we deduce that
\begin{eqnarray*}
\mathrm{ord}_q~P(j(\tau))&=&
\mathrm{ord}_q~\mathbf{N}_{\mathcal{F}_N/\mathcal{F}_1}(g(\tau))\\
&\geq& -(N/2)\cdot[\mathcal{F}_N:\mathcal{F}_1]\\
&=& -(N/2)\cdot\#\mathrm{GL}_2(\mathbb{Z}/N\mathbb{Z})/\{\pm I_2\}
\quad\textrm{by (\ref{Gal(F_N/F_1)})}\\
&=&\left\{\begin{array}{ll} -6 & \textrm{if}~N=2,\vspace{0.1cm}\\
-(N^5/4)\prod_{p|N}(1-p^{-1})(1-p^{-2}) & \textrm{if}~N\geq3
\end{array}\right.\quad\textrm{by (\ref{size})}.
\end{eqnarray*}
Thus we get from the fact $\mathrm{ord}_q~j(\tau)=-1$ that
\begin{equation*}
\#Z_{\left[\begin{smallmatrix}s\\t\end{smallmatrix}\right]}\leq\deg P(X)\leq \left\{\begin{array}{ll}
6 & \textrm{if}~N=2,\vspace{0.1cm}\\
(N^5/4)\prod_{p|N}(1-p^{-1})(1-p^{-2}) & \textrm{if}~N\geq3.
\end{array}\right.
\end{equation*}
And, if we let
\begin{equation*}
Z=\bigcup_{\left[\begin{smallmatrix}s\\t\end{smallmatrix}\right]\in S}Z_{\left[\begin{smallmatrix}s\\t\end{smallmatrix}\right]},
\end{equation*}
then
\begin{eqnarray*}
\#Z\leq\sum_{\left[\begin{smallmatrix}s\\t\end{smallmatrix}\right]\in S}
\#Z_{\left[\begin{smallmatrix}s\\t\end{smallmatrix}\right]} &\leq& \#S\cdot \max_{\left[\begin{smallmatrix}s\\t\end{smallmatrix}\right]\in
S}\{\#Z_{\left[\begin{smallmatrix}s\\t\end{smallmatrix}\right]}\}\\
&\leq& \left\{\begin{array}{ll}
12 & \textrm{if}~N=2,\vspace{0.1cm}\\
((N+1)[N/2]-1)(N^5/4)\prod_{p|N}(1-p^{-1})(1-p^{-2}) &
\textrm{if}~N\geq3.
\end{array}\right.
\end{eqnarray*}
\par
Now, let $K$ be an imaginary quadratic field lying outside $Z$. Then
the singular value $g_{\left[\begin{smallmatrix}0\\1/N\end{smallmatrix}\right]}(\tau_K)^{12N}$ generates $K_{(N)}$
over $H_K$. Indeed, suppose that it does not generate $K_{(N)}$ over
$H_K$. Then there exists a non-identity element
$\alpha=\left[\begin{matrix}t-Bs &
-Cs\\s&t\end{matrix}\right]$ of
$W_{N,\tau_K}/\mathrm{Ker}_{N,\tau_K}$
($\simeq\mathrm{Gal}(K_{(N)}/H_K)$) in Proposition \ref{Gee} which
fixes $g_{\left[\begin{smallmatrix}0\\1/N\end{smallmatrix}\right]}(\tau_K)^{12N}$. Here we may assume that $\left[\begin{matrix}s\\t\end{matrix}\right]$
belongs to $S$ because $W_{N,\tau_K}/\mathrm{Ker}_{N,\tau_K}$ is
a subgroup or a quotient of
$\mathrm{GL}_2(\mathbb{Z}/N\mathbb{Z})/\{\pm I_2\}$. We derive that
\begin{eqnarray*}
0&=&g_{\left[\begin{smallmatrix}0\\1/N\end{smallmatrix}\right]}(\tau_K)^{12N}
-(g_{\left[\begin{smallmatrix}0\\1/N\end{smallmatrix}\right]}(\tau_K)^{12N})^\alpha\\
&=& g_{\left[\begin{smallmatrix}0\\1/N\end{smallmatrix}\right]}(\tau_K)^{12N}
-(g_{\left[\begin{smallmatrix}0\\1/N\end{smallmatrix}\right]}(\tau)^{12N})^\alpha(\tau_K)\quad\textrm{by
Proposition \ref{Gee}}\\
&=&g_{\left[\begin{smallmatrix}0\\1/N\end{smallmatrix}\right]}(\tau_K)^{12N}-
g_{\left[\begin{smallmatrix}s/N\\t/N\end{smallmatrix}\right]}(\tau_K)^{12N}\quad\textrm{by Proposition
\ref{transform} (iv) and (v)}.
\end{eqnarray*}
But this implies that $K$ belongs to $Z_{\left[\begin{smallmatrix}s\\t\end{smallmatrix}\right]}$ ($\subseteq Z$),
which yields a contradiction. Therefore we conclude that
\begin{equation*}
\{\textrm{imaginary quadratic fields
$K$}~|~\textrm{$g_{\left[\begin{smallmatrix}0\\1/N\end{smallmatrix}\right]}(\tau_K)^{12N}$ does not generate
$K_{(N)}$ over $H_K$}\}\subseteq Z.
\end{equation*}
This completes the proof.
\end{proof}

Let $K$ be an imaginary quadratic field and $\mathfrak{f}$ be a nontrivial
ideal of $\mathcal{O}_K$. For a character $\chi$ of
$\mathrm{Cl}(\mathfrak{f})$ we let $\mathfrak{f}_\chi$ be the
conductor of $\chi$ and $\chi_0$ be the proper character of
$\mathrm{Cl}(\mathfrak{f}_\chi)$ corresponding to $\chi$. If
$\mathfrak{f}\neq\mathcal{O}_K$ and $\chi$ is also a nontrivial
character of $\mathrm{Cl}(\frak{f})$, then we define the
\textit{Stickelberger element}
\begin{equation*} S_\mathfrak{f}(\chi,g_\mathfrak{f})
=\sum_{C\in\mathrm{Cl}(\mathfrak{f})}
\chi(C)\log|g_\mathfrak{f}(C)|,
\end{equation*}
and the \textit{$L$-function}
\begin{equation*}
L_\mathfrak{f}(s,\chi)=\sum_{\mathfrak{a}} \frac{\chi([\mathfrak{a}])}{\mathbf{N}_{K/\mathbb{Q}}(\mathfrak{a})^s}\quad(s\in\mathbb{C}),
\end{equation*}
where $\mathfrak{a}$ runs over all nontrivial ideals of $\mathcal{O}_K$
relatively prime to $\mathfrak{f}$ and $[\mathfrak{a}]$ is the class containing $\mathfrak{a}$.

\begin{proposition}[The second Kronecker limit formula]\label{LandS}
If $\mathfrak{f}_\chi\neq\mathcal{O}_K$, then we have
\begin{equation*}
L_{\mathfrak{f}_\chi}(1,\chi_0)
\prod_{\mathfrak{p}\,|\,\mathfrak{f},~\mathfrak{p}\,\nmid\,\mathfrak{f}_\chi}
(1-\overline{\chi}_0([\mathfrak{p}]))=-\frac{\pi\chi_0([\gamma\mathfrak{d}_K\mathfrak{f}_\chi])}{3N(\mathfrak{f}_\chi)\sqrt{-d_K}\omega(\mathfrak{f}_\chi)
T_\gamma(\overline{\chi}_0)}~S_{\mathfrak{f}}(\overline{\chi},g_\mathfrak{f}),
\end{equation*}
where $\mathfrak{d}_K$ is
the different of $K/\mathbb{Q}$, $\gamma$ is a nonzero element of $K$ so that
$\gamma\mathfrak{d}_K\mathfrak{f}_\chi$ becomes an ideal of $\mathcal{O}_K$ relatively prime to $\mathfrak{f}_\chi$,
$N(\mathfrak{f}_\chi)$ is the smallest positive integer in $\mathfrak{f}_\chi$,
 $\omega(\mathfrak{f}_\chi)
=|\{\zeta\in\mathcal{O}_K^\times~|~\zeta\equiv1\Mod{\mathfrak{f}_\chi}\}|$ and
\begin{eqnarray*}
T_\gamma(\overline{\chi}_0)=
\sum_{x+\mathfrak{f}_\chi\in\pi_{\mathfrak{f}_\chi}(\mathcal{O}_K)^\times}
\overline{\chi}_0([x\mathcal{O}_K])e^{2\pi
i\mathrm{Tr}_{K/\mathbb{Q}}(x\gamma)}.
\end{eqnarray*}
\end{proposition}
\begin{proof}
See \cite[Chapter 22, Theorems 1 and 2]{Lang} and \cite[Chapter 11, Theorem
2.1]{K-L}.
\end{proof}

\begin{remark}\label{remarkLandS}
\begin{itemize}
\item[(i)]
The Euler factor
$\prod_{\mathfrak{p}\,|\,\mathfrak{f},~\mathfrak{p}\,\nmid\,\mathfrak{f}_\chi}
(1-\overline{\chi}_0([\mathfrak{p}]))$ is understood to be $1$ if
there is no prime ideal $\mathfrak{p}$ such that
$\mathfrak{p}\,|\,\mathfrak{f}$ and
$\mathfrak{p}\nmid\mathfrak{f}_\chi$.
\item[(ii)] As is well-known,  $L_{\mathfrak{f}_\chi}(1,\chi_0)\neq0$
\cite[Chapter IV, Proposition 5.7]{Janusz}.
\end{itemize}
\end{remark}

\begin{theorem}\label{overK}
Let $K$ be an imaginary quadratic field and $\mathfrak{f}$ be a
nontrivial proper ideal of $\mathcal{O}_K$ whose prime ideal factorization is given by
\begin{equation*}
\mathfrak{f}=\prod_{k=1}^n \mathfrak{p}_k^{e_k}.
\end{equation*}
Assume that
\begin{equation}\label{hypothesis}
[K_\mathfrak{f}:K]>2\sum_{k=1}^n
[K_{\mathfrak{f}\mathfrak{p}_k^{-e_k}}:K].
\end{equation}
Then $g_\mathfrak{f}(C_0)$ generates $K_\mathfrak{f}$ over $K$.
\end{theorem}
\begin{proof}
Set $F=K(g_\mathfrak{f}(C_0))$. We then derive that
\begin{eqnarray}\label{number1}
&&\#\{\textrm{characters}~\chi~\mathrm{of}~\mathrm{Gal}(K_\mathfrak{f}/K)~|~
\chi|_{\mathrm{Gal}(K_\mathfrak{f}/F)}\neq1\}\nonumber\\
&=&\#\{\textrm{characters}~\chi~
\textrm{of}~\mathrm{Gal}(K_\mathfrak{f}/K)\}
-\#\{\textrm{characters}~\chi~ \textrm{of}~\mathrm{Gal}(F/K)\}\nonumber\\
&=&[K_\mathfrak{f}:K]-[F:K].
\end{eqnarray}
Furthermore, we have
\begin{eqnarray}\label{number2}
&&\#\{\textrm{characters}~\chi~\textrm{of}~\mathrm{Gal}(K_\mathfrak{f}/K)
~|~\mathfrak{p}_k\nmid\mathfrak{f}_\chi~ \textrm{for
some}~k\}\nonumber\\
&=&\#\{\textrm{characters}~\chi~\textrm{of}~\mathrm{Gal}(K_\mathfrak{f}/K)
~|~\mathfrak{f}_\chi\,|\,\mathfrak{f}\mathfrak{p}_k^{-e_k}
~\textrm{for some}~k\}\nonumber\\
&\leq&\sum_{k=1}^n
\#\{\textrm{characters}~\chi~\textrm{of}~\mathrm{Gal}(K_{\mathfrak{f}\mathfrak{p}_k^{-e_k}}/K)\}\nonumber\\
&=&\sum_{k=1}^n[K_{\mathfrak{f}\mathfrak{p}_k^{-e_k}}:K].
\end{eqnarray}
\par
Now, suppose that $F$ is properly contained in $K_\mathfrak{f}$.
Then we get by the assumption (\ref{hypothesis}) that
\begin{eqnarray*}
[K_\mathfrak{f}:K]-[F:K]&=&[K_\mathfrak{f}:K]
(1-1/[K_\mathfrak{f}:F])\\
&>&2\sum_{k=1}^n
[K_{\mathfrak{f}\mathfrak{p}_k^{-e_k}}:K] (1-1/2)\\
&=&\sum_{k=1}^n
[K_{\mathfrak{f}\mathfrak{p}_k^{-e_k}}:K].
\end{eqnarray*}
This, together with (\ref{number1}) and (\ref{number2}), implies
that there exists a character $\chi$ of
$\mathrm{Gal}(K_\mathfrak{f}/K)$ such that
\begin{equation}\label{char1}
\chi|_{\mathrm{Gal}(K_\mathfrak{f}/F)}\neq1,
\end{equation}
\begin{equation}\label{char2}
\mathfrak{p}_k\,|\,\mathfrak{f}_\chi\quad\textrm{for all
$k=1,\ldots,n$}.
\end{equation}
Identifying $\mathrm{Cl}(\mathfrak{f})$ and
$\mathrm{Gal}(K_\mathfrak{f}/K)$ via the Artin reciprocity map, we obtain from
Proposition \ref{LandS} and (\ref{char2}) that
\begin{equation}\label{contradiction}
0\neq
L_{\mathfrak{f}_\chi}(1,\chi_0)=
-\frac{\pi\chi_0([\gamma\mathfrak{d}_K\mathfrak{f}_\chi])}{3N(\mathfrak{f}_\chi)\sqrt{-d_K}\omega(\mathfrak{f}_\chi)
T_\gamma(\overline{\chi}_0)}~S_{\mathfrak{f}}(\overline{\chi},g_\mathfrak{f}).
\end{equation}
On the other hand, we achieve that
\begin{eqnarray*}
S_\mathfrak{f}(\overline{\chi},g_\mathfrak{f})
&=&\sum_{C\in\mathrm{Cl}(\mathfrak{f})}\overline{\chi}(C)\log|g_\mathfrak{f}(C_0)^C|
\quad\textrm{by (\ref{Artin})}\\
&=&\sum_{\begin{smallmatrix}C_1\in\mathrm{Gal}(K_\mathfrak{f}/K)\\C_1
\Mod{\mathrm{Gal}(K_\mathfrak{f}/F)}\end{smallmatrix}}
\sum_{C_2\in\mathrm{Gal}(K_\mathfrak{f}/F)}
\overline{\chi}(C_1C_2)\log|g_\mathfrak{f}(C_0)^{C_1C_2}|\\
&=&\sum_{C_1}\sum_{C_2}\overline{\chi}(C_1)\overline{\chi}(C_2)\log|(g_\mathfrak{f}(C_0)^{C_2})^{C_1}|\\
&=&\sum_{C_1}\overline{\chi}(C_1)\log|g_\mathfrak{f}(C_0)^{C_1}|(\sum_{C_2}
\overline{\chi}(C_2))
\quad\textrm{by the fact}~g_\mathfrak{f}(C_0)\in F\\
&=&0\quad\textrm{by (\ref{char1})},
\end{eqnarray*}
which contradicts (\ref{contradiction}). Therefore, we conclude
$F=K_\mathfrak{f}$ as desired.
\end{proof}

\begin{remark}\label{degree}
\begin{itemize}
\item[(i)]
For a nontrivial integral ideal $\mathfrak{f}$ of an imaginary
quadratic field $K$, we have a degree formula
\begin{equation}\label{deg}
[K_{\mathfrak{f}}:K]=\frac{h_K\varphi(\mathfrak{f})\omega(\mathfrak{f})}{\omega_K},
\end{equation}
where $\varphi$ is the (multiplicative) Euler function for ideals, namely
\begin{equation*}
\varphi(\mathfrak{p}^n)=(\mathbf{N}_{K/\mathbb{Q}}(\mathfrak{p})-1)\mathbf{N}_{K/\mathbb{Q}}(\mathfrak{p})^{n-1}
\end{equation*}
for a prime ideal power $\mathfrak{p}^n$ ($n\geq1$),
$\omega(\mathfrak{f})$ is the number of roots of unity in $K$ which
are $\equiv1\Mod{\mathfrak{f}}$ and $\omega_K$ is the number of
roots of unity in $K$ \cite[Chapter VI, Theorem 1]{Lang2}.
\par
Let $N$ ($\geq2$) be an integer whose prime factorization is given
by
\begin{equation*} N=\prod_{a=1}^A p_a^{u_a}\prod_{b=1}^B
q_b^{v_b} \prod_{c=1}^C r_c^{w_c}\quad(A,B,C,u_a,v_b,w_c\geq0),
\end{equation*}
where each $p_a$ (respectively, $q_b$ and $r_c$) splits
(respectively, is inert and ramified) in $K$. One can then verify
that the condition
\begin{equation*}
4\sum_{a=1}^{A}\frac{1}{(p_a-1)p_a^{u_a-1}}
+2\sum_{b=1}^B\frac{1}{(q_b^2-1)q_b^{2(v_b-1)}}
+2\sum_{c=1}^C\frac{1}{(r_c-1)r_c^{2w_c-1}}<\frac{\omega((N))}{\omega_K}
\end{equation*}
implies the assumption (\ref{hypothesis}) when $\mathfrak{f}=(N)$.
\item[(ii)]
Let $d_k$ ($k=1,\ldots,n$) be the exponent of the group
$(\mathcal{O}_K/\mathfrak{p}_k^{e_k})^\times$.
Schertz \cite[Theorem 3]{Schertz} proved that
if the conductor of the extension $K_\mathfrak{f}/K$ is exactly $\mathfrak{f}$, then
$g_\mathfrak{f}(C_0)$ is a primitive
generator of $K_\mathfrak{f}$ over $K$ in either case when $n=1$ or
\begin{equation}\label{Schertzcondition}
d_k\nmid2~(k=1,\ldots,n-1),~d_n\nmid 2\omega_K~\textrm{and}~
\mathfrak{p}_n^{e_n}\nmid\gcd(6,\omega_K).
\end{equation}
Note that we do not require any condition on the conductor of the extension $K_\mathfrak{f}/K$.
\end{itemize}
\end{remark}

\section{Siegel-Ramachandra invariants of conductor $2$}\label{sec4}

Throughout this section we let $K$ be an imaginary quadratic field.
We shall examine certain properties of the singular value
$g_{\left[\begin{smallmatrix}0\\1/2\end{smallmatrix}\right]}(\tau_K)$ which is a $24$-th root of
$g_{(2)}(C_0)$. Although most of the results here are classical and
known, we will present relatively short and new proofs purely in
terms of Siegel functions.
\par
By the definition (\ref{Siegel}) we have
\begin{equation}\label{Siegel2}
\begin{array}{lll}
g_{\left[\begin{smallmatrix}0\\1/2\end{smallmatrix}\right]}(\tau)&=&
2\zeta_4q^{1/12}\displaystyle\prod_{n=1}^\infty(1+q^n)^{2},\\
 g_{\left[\begin{smallmatrix}1/2\\0\end{smallmatrix}\right]}(\tau)&=&
-q^{-1/24}\displaystyle\prod_{n=1}^\infty(1-q^{n-1/2})^{2},\\
 g_{\left[\begin{smallmatrix}1/2\\1/2\end{smallmatrix}\right]}(\tau)&=&
\zeta_8^3q^{-1/24}\displaystyle\prod_{n=1}^\infty(1+q^{n-1/2})^{2}.
\end{array}
\end{equation}
Let $\gamma_2(\tau)$ be the cube root of $j(\tau)$ whose Fourier
expansion begins with the term $q^{-1/3}$.

\begin{lemma}\label{jtog}
\begin{itemize}
\item[\textup{(i)}] We have the identity
\begin{equation*}
g_{\left[\begin{smallmatrix}0\\1/2\end{smallmatrix}\right]}(\tau)
g_{\left[\begin{smallmatrix}1/2\\0\end{smallmatrix}\right]}(\tau) g_{\left[\begin{smallmatrix}1/2\\1/2\end{smallmatrix}\right]}(\tau)=2\zeta_8.
\end{equation*}
\item[\textup{(ii)}]  We have the relations
\begin{equation*}
\gamma_2(\tau)=
\frac{g_{\left[\begin{smallmatrix}0\\1/2\end{smallmatrix}\right]}(\tau)^{12}+16}
{g_{\left[\begin{smallmatrix}0\\1/2\end{smallmatrix}\right]}(\tau)^4}=
\frac{g_{\left[\begin{smallmatrix}1/2\\0\end{smallmatrix}\right]}(\tau)^{12}+16}
{g_{\left[\begin{smallmatrix}1/2\\0\end{smallmatrix}\right]}(\tau)^4}=
\frac{g_{\left[\begin{smallmatrix}1/2\\1/2\end{smallmatrix}\right]}(\tau)^{12}+16} {g_{\left[\begin{smallmatrix}1/2\\1/2\end{smallmatrix}\right]}(\tau)^4}.
\end{equation*}
\end{itemize}
\end{lemma}
\begin{proof}
(i) We obtain by (\ref{Siegel2})
\begin{equation*}
g_{\left[\begin{smallmatrix}0\\1/2\end{smallmatrix}\right]}(\tau)
g_{\left[\begin{smallmatrix}1/2\\0\end{smallmatrix}\right]}(\tau)
g_{\left[\begin{smallmatrix}1/2\\1/2\end{smallmatrix}\right]}(\tau)=2\zeta_8\prod_{n=1}^\infty
(1+q^n)^2(1-q^{2n-1})^2=2\zeta_8\prod_{n=1}^\infty\frac{(1-q^{2n})^2}{(1-q^n)^2}\cdot
\frac{(1-q^n)^2}{(1-q^{2n})^2}=2\zeta_8.
\end{equation*}
(ii) Since $g_{\left[\begin{smallmatrix}0\\1/2\end{smallmatrix}\right]}(\tau)^{12}\in\mathcal{F}_2$ by Proposition
\ref{level} and
\begin{equation*}
\mathrm{Gal}(\mathcal{F}_2/\mathcal{F}_1)\simeq
\mathrm{GL}_2(\mathbb{Z}/2\mathbb{Z})/\{\pm I_2\} =\bigg\{
\left[\begin{matrix}1&0\\0&1
\end{matrix}\right],
\left[\begin{matrix}1&1\\0&1
\end{matrix}\right],
\left[\begin{matrix}1&1\\1&0
\end{matrix}\right],
\left[\begin{matrix}1&0\\1&1
\end{matrix}\right],
\left[\begin{matrix}0&1\\1&0
\end{matrix}\right],
\left[\begin{matrix}0&1\\1&1
\end{matrix}\right]\bigg\}
\end{equation*}
by (\ref{Gal(F_N/F_1)}), we derive that
\begin{eqnarray*}
&&\prod_{\sigma\in\mathrm{Gal}(\mathcal{F}_2/\mathcal{F}_1)}
(X-(g_{\left[\begin{smallmatrix}0\\1/2\end{smallmatrix}\right]}(\tau)^{12})^\sigma)\\
&=&(X-g_{\left[\begin{smallmatrix}0\\1/2\end{smallmatrix}\right]}(\tau)^{12})^2 (X-g_{\left[\begin{smallmatrix}1/2\\0\end{smallmatrix}\right]}(\tau)^{12})^2
(X-g_{\left[\begin{smallmatrix}1/2\\1/2\end{smallmatrix}\right]}(\tau)^{12})^2\quad\textrm{by Proposition \ref{transform} (iv) and (v)}\\
&=&(X^3+48X^2+(-q^{-1}+24-196884q+\cdots)X+4096)^2\quad\textrm{by (\ref{Siegel2})}\\
&=&(X^3+48X^2+(-j(\tau)+768)X+4096)^2\quad\textrm{by (\ref{jFourier})}\\
&=&((X+16)^3-j(\tau)X)^2\\
&=&((X+16)^3-\gamma_2(\tau)^3X)^2.
\end{eqnarray*}
Hence we get
\begin{equation*}
\gamma_2(\tau)=\xi_1\frac{g_{\left[\begin{smallmatrix}0\\1/2\end{smallmatrix}\right]}(\tau)^{12}+16}
{g_{\left[\begin{smallmatrix}0\\1/2\end{smallmatrix}\right]}(\tau)^{4}}
=\xi_2\frac{g_{\left[\begin{smallmatrix}1/2\\0\end{smallmatrix}\right]}(\tau)^{12}+16}
{g_{\left[\begin{smallmatrix}1/2\\0\end{smallmatrix}\right]}(\tau)^{4}}
=\xi_3\frac{g_{\left[\begin{smallmatrix}1/2\\1/2\end{smallmatrix}\right]}(\tau)^{12}+16} {g_{\left[\begin{smallmatrix}1/2\\1/2\end{smallmatrix}\right]}(\tau)^{4}}
\end{equation*}
for some cube roots of unity $\xi_k$ ($k=1,2,3$). Comparing the
leading terms of Fourier expansions we conclude
$\xi_1=\xi_2=\xi_3=1$.
\end{proof}
\begin{remark}
Let
\begin{equation*}
\eta(\tau)=\sqrt{2\pi}\zeta_8q^{1/24}\prod_{n=1}^\infty(1-q^n)
\end{equation*}
be the Dedekind eta function, and
\begin{equation*}
\mathfrak{f}(\tau)=\zeta_{48}^{-1}\frac{\eta((\tau+1)/2)}{\eta(\tau)},\quad
\mathfrak{f}_1(\tau)=\frac{\eta(\tau/2)}{\eta(\tau)},\quad
\mathfrak{f}_2(\tau)=\sqrt{2}\frac{\eta(2\tau)}{\eta(\tau)}
\end{equation*}
be the Weber functions. Then one can deduce the following identities
\begin{equation}\label{ftog}
\mathfrak{f}(\tau)^2=\zeta_8^5g_{\left[\begin{smallmatrix}1/2\\1/2\end{smallmatrix}\right]}(\tau),\quad
\mathfrak{f}_1(\tau)^2=-g_{\left[\begin{smallmatrix}1/2\\0\end{smallmatrix}\right]}(\tau),\quad
\mathfrak{f}_2(\tau)^2=\zeta_4^3g_{\left[\begin{smallmatrix}0\\1/2\end{smallmatrix}\right]}(\tau),
\end{equation}
and hence Lemma \ref{jtog} (ii) can be reformulated in terms of the
Weber functions as in the classical case \cite[Theorem
12.17]{Cox}.
\end{remark}

\begin{lemma}\label{intcoeff}
If $x$ is a real algebraic integer, then $\min(x,K)$ has integer
coefficients.
\end{lemma}
\begin{proof}
Since $x\in\mathbb{R}$, we get
\begin{equation*}
[K(x):K]=\frac{[K(x):\mathbb{Q}(x)]\cdot[\mathbb{Q}(x):\mathbb{Q}]}{[K:\mathbb{Q}]}
=[\mathbb{Q}(x):\mathbb{Q}],
\end{equation*}
from which it follows that $\min(x,K)=\min(x,\mathbb{Q})$.
Furthermore, $\min(x,K)$ has integer coefficients, because $x$ is an
algebraic integer.
\end{proof}

\begin{proposition}\label{j2}
Let $K$ be an imaginary quadratic field of discriminant $d_K$.
\begin{itemize}
\item[\textup{(i)}] $j(\tau_K)$ is a real algebraic integer which generates
$H_K$ over $K$.
\item[\textup{(ii)}] If $p$ is a prime dividing the discriminant of
$\min(j(\tau_K),K)$, then $(\frac{d_K}{p})\neq1$ and
$p\leq|d_K|$.
\end{itemize}
\end{proposition}
\begin{proof}
(i) See \cite[Chapter 5, Theorem 4 and Chapter 10, Theorem 1]{Lang}.\\
(ii) See \cite{G-Z}, \cite{Dorman} or \cite[Theorem 13.28]{Cox}.
\end{proof}

\begin{remark}\label{algint}
For any $\left[\begin{matrix}r_1\\r_2\end{matrix}\right]\in\mathbb{Q}^2\setminus\mathbb{Z}^2$,
$g_{\left[\begin{smallmatrix}r_1\\r_2\end{smallmatrix}\right]}(\tau_K)$ is an algebraic integer by Propositions
\ref{transform} (vi) and \ref{j2} (i).
\end{remark}

\begin{theorem}\label{properties}
Let $K\neq\mathbb{Q}(\sqrt{-3}),\mathbb{Q}(\sqrt{-1})$ and set
$x=\mathbf{N}_{K_{(2)}/H_K}(g_{\left[\begin{smallmatrix}0\\1/2\end{smallmatrix}\right]}(\tau_K)^{12})$. Assume
that $2$ is not inert in $K$ \textup{(}equivalently,
$d_K\equiv0\Mod{4}$ or $d_K\equiv1\Mod{8}$\textup{)}.
\begin{itemize}
\item[\textup{(i)}]
$x$ generates $H_K$ over $K$.
\item[\textup{(ii)}]
$x$ is a real algebraic integer dividing $2^{12}$ whose minimal
polynomial $\min(x,K)$ has integer coefficients.
\item[\textup{(iii)}] If $p$ is an odd prime dividing the discriminant of
 $\min(x,K)$, then $(\frac{d_K}{p})\neq1$ and $p\leq|d_K|$.
\end{itemize}
\end{theorem}
\begin{proof}
(i) Since $g_{\left[\begin{smallmatrix}0\\1/2\end{smallmatrix}\right]}(\tau)^{12}\in\mathcal{F}_2$,
$g_{\left[\begin{smallmatrix}0\\1/2\end{smallmatrix}\right]}(\tau_K)^{12}$ lies in $K_{(2)}$ by Proposition
\ref{CM}. Moreover, we have
\begin{equation*}
[K_{(2)}:H_K]=\left\{\begin{array}{ll} 2 &
\textrm{if}~d_K\equiv0\Mod{4}, \vspace{0.1cm}\\
1 & \textrm{if}~d_K\equiv1\Mod{8}
\end{array}\right.
\end{equation*}
by the degree formula (\ref{deg}), and
\begin{equation*}
\mathrm{Gal}(K_{(2)}/H_K)\simeq
W_{2,\tau_K}/\mathrm{Ker}_{2,\tau_K}=\left\{\begin{array}{ll}
\bigg\{\left[\begin{matrix}1&0\\0&1\end{matrix}\right],
\left[\begin{matrix}1&0\\1&1\end{matrix}\right]\bigg\}
& \textrm{if}~d_K\equiv0\Mod{8},
\vspace{0.1cm}\\
\bigg\{\left[\begin{matrix}1&0\\0&1\end{matrix}\right],
\left[\begin{matrix}0&1\\1&0\end{matrix}\right]\bigg\}
& \textrm{if}~d_K\equiv4\Mod{8},\vspace{0.1cm}\\
\bigg\{
\left[\begin{matrix}1&0\\0&1\end{matrix}\right]\bigg\} &
\textrm{if}~d_K\equiv1\Mod{8}
\end{array}\right.
\end{equation*} by Proposition \ref{Gee}.
So we obtain
\begin{equation}\label{norm}
x=\mathbf{N}_{K_{(2)}/H_K}(g_{\left[\begin{smallmatrix}0\\1/2\end{smallmatrix}\right]}(\tau_K)^{12})=\left\{
\begin{array}{ll}
g_{\left[\begin{smallmatrix}0\\1/2\end{smallmatrix}\right]}(\tau_K)^{12} g_{\left[\begin{smallmatrix}1/2\\1/2\end{smallmatrix}\right]}(\tau_K)^{12}
 &
\textrm{if}~d_K\equiv0\Mod{8},
\vspace{0.1cm}\\
g_{\left[\begin{smallmatrix}0\\1/2\end{smallmatrix}\right]}(\tau_K)^{12} g_{\left[\begin{smallmatrix}1/2\\0\end{smallmatrix}\right]}(\tau_K)^{12}
 &
\textrm{if}~d_K\equiv4\Mod{8},\vspace{0.1cm}\\
g_{\left[\begin{smallmatrix}0\\1/2\end{smallmatrix}\right]}(\tau_K)^{12} & \textrm{if}~d_K\equiv1\Mod{8}
\end{array}\right.
\end{equation}
by Propositions \ref{Gee} and \ref{transform} (iv), (v); and hence
\begin{equation}\label{singular}
j(\tau_K)=\left\{\begin{array}{ll} (256-x)^3/x^2 &
\textrm{if}~d_K\equiv0\Mod{4}, \vspace{0.1cm}\\
(x+16)^3/x & \textrm{if}~d_K\equiv1\Mod{8}
\end{array}\right.
\end{equation}
by Lemma \ref{jtog}. Therefore $x$ generates $H_K$ over $K$ by
Proposition \ref{j2} (i).\\
(ii) We see that $x\in\mathbb{R}$ by the definition (\ref{theta}),
(\ref{Siegel2}) and (\ref{norm}). Furthermore, since $x$ is an
algebraic integer by Remark \ref{algint}, $\min(x,K)$ has integer
coefficients by Lemma \ref{intcoeff}. And, $x$ divides $2^{12}$ by (\ref{norm}) and Lemma \ref{jtog} (i).\\
(iii) If $h_K=1$, there is nothing to prove. So we assume $h_K>1$.
If $\sigma_1$ and $\sigma_2$ are distinct elements of
$\mathrm{Gal}(H_K/K)$, then we derive from (\ref{singular}) that
\begin{eqnarray*}
&&j(\tau_K)^{\sigma_1}-j(\tau_K)^{\sigma_2}\\&=&
\left\{\begin{array}{ll}
(x_1-x_2)(-x_1^2x_2^2+196608x_1x_2-16777216x_1-16777216x_2)/x_1^2x_2^2
& \textrm{if}~d_K\equiv0\Mod{4},
\vspace{0.1cm}\\
(x_1-x_2)(x_1^2x_2+x_1x_2^2+48x_1x_2-4096)/x_1x_2 &
\textrm{if}~d_K\equiv1\Mod{8},
\end{array}\right.
\end{eqnarray*}
where $x_1=x^{\sigma_1}$ and $x_2=x^{\sigma_2}$. Observe from (ii)
that there is no prime ideal $\mathfrak{p}$ of $H_K$ which contains
$x_1x_2$ and lies above an odd prime. Therefore, if $p$ is an odd
prime dividing the discriminant of $\min(x,K)$, then $(\frac{d_K}{p})\neq1$
and $p\leq|d_K|$ by Proposition \ref{j2} (ii).
\end{proof}

\begin{remark}\label{inertconj}
Let $K\neq\mathbb{Q}(\sqrt{-3}),\mathbb{Q}(\sqrt{-1})$. Since
$g_{\left[\begin{smallmatrix}0\\1/2\end{smallmatrix}\right]}(\tau_K)^{24}$ generates $K_{(2)}$ over $K$ by Theorem
\ref{overK} and Remark \ref{degree}, so does
$g_{\left[\begin{smallmatrix}0\\1/2\end{smallmatrix}\right]}(\tau_K)^{12}$. If $2$ is inert in $K$, then
\begin{equation*}
\mathrm{Gal}(K_{(2)}/H_K)\simeq
W_{2,\tau_K}/\mathrm{Ker}_{2,\tau_K}= \bigg\{
\left[\begin{matrix}
1&0\\0&1
\end{matrix}\right],
\left[\begin{matrix}
1&1\\1&0
\end{matrix}\right],
\left[\begin{matrix}
0&1\\1&1
\end{matrix}\right]
\bigg\}
\end{equation*}
by Proposition \ref{Gee}. And, we derive from Proposition
\ref{transform} (iv), (v) and Lemma  \ref{jtog} (i) that
\begin{equation*}
\mathbf{N}_{K_{(2)}/H_K}(g_{\left[\begin{smallmatrix}0\\1/2\end{smallmatrix}\right]}(\tau_K)^{12})=
g_{\left[\begin{smallmatrix}0\\1/2\end{smallmatrix}\right]}(\tau_K)^{12}
g_{\left[\begin{smallmatrix}1/2\\0\end{smallmatrix}\right]}(\tau_K)^{12}
g_{\left[\begin{smallmatrix}1/2\\1/2\end{smallmatrix}\right]}(\tau_K)^{12}=
-2^{12}.
\end{equation*}
Therefore, in this case one cannot develop a theory like Theorem
\ref{properties} with
$\mathbf{N}_{K_{(2)}/H_K}(g_{\left[\begin{smallmatrix}0\\1/2\end{smallmatrix}\right]}(\tau_K)^{12})$.
\end{remark}

\begin{theorem}\label{generator}
Let $K$ be an imaginary quadratic field of discriminant $d_K$.
Assume that $2$ is inert and $3$ is not ramified in $K$
\textup{(}equivalently, $d_K\equiv5\Mod{8}$ and
$d_K\not\equiv0\Mod{3}$\textup{)}.
\begin{itemize}
\item[\textup{(i)}] The real algebraic integer
$\zeta_8g_{\left[\begin{smallmatrix}0\\1/2\end{smallmatrix}\right]}(\tau_K)$ generates $K_{(2)}$ over $H_K$.
\item[\textup{(ii)}] The real algebraic integer $\gamma_2(\tau_K)$ generates $H_K$ over $K$.
\end{itemize}
\end{theorem}
\begin{proof}
(i) Let $\alpha=g_{\left[\begin{smallmatrix}0\\1/2\end{smallmatrix}\right]}(\tau_K)$. It is an algebraic integer
by Remark \ref{algint}, and $\zeta_8\alpha\in\mathbb{R}$ by the
definition (\ref{theta}) and (\ref{Siegel2}). Since $\alpha^4$ is a
real cube root of $\alpha^{12}$, we get from Remark \ref{inertconj}
that
\begin{equation*}
[K_{(2)}(\alpha^4):K_{(2)}]=[K(\alpha^4):K(\alpha^{12})]
=\frac{[K(\alpha^4):\mathbb{Q}(\alpha^{4})]
[\mathbb{Q}(\alpha^4):\mathbb{Q}(\alpha^{12})]}{[K(\alpha^{12}):\mathbb{Q}(\alpha^{12})]}
=[\mathbb{Q}(\alpha^4):\mathbb{Q}(\alpha^{12})]=1~\textrm{or}~3.
\end{equation*}
Furthermore, since $g_{\left[\begin{smallmatrix}0\\1/2\end{smallmatrix}\right]}(\tau)^4\in\mathcal{F}_6$ by
Proposition \ref{level}, we get $\alpha^4\in K_{(6)}$ by Proposition
\ref{CM}, from which it follows that $[K_{(2)}(\alpha^4):K_{(2)}]$
divides
\begin{equation*}
[K_{(6)}:K_{(2)}]=\left\{\begin{array}{ll}
2 & \textrm{if $3$ splits in $K$},\vspace{0.1cm}\\
4 & \textrm{if $3$ is inert in $K$}\end{array}\right.
\end{equation*}
by the degree formula (\ref{deg}). Hence
$[K_{(2)}(\alpha^4):K_{(2)}]=1$, which implies $\alpha^4\in
K_{(2)}$.
\par
On the other hand, since
$\zeta_8^{-1}g_{\left[\begin{smallmatrix}0\\1/2\end{smallmatrix}\right]}(\tau)^3\in\mathcal{F}_8$ by Proposition
\ref{level}, we obtain $\zeta_8^{-1}\alpha^3\in K_{(8)}$ by
Proposition \ref{CM}. One can then readily check by Proposition
\ref{Gee} that
\begin{equation*}
\mathrm{Gal}(K_{(8)}/K_{(2)})\simeq\bigg\langle
\left[\begin{matrix}5&4\\4&1\end{matrix}\right]\bigg\rangle\times
\left\{\begin{array}{ll} \bigg\langle
\left[\begin{matrix}7&6\\2&1\end{matrix}\right]
\bigg\rangle &
\textrm{if}~d_K\equiv5\Mod{16},\vspace{0.1cm}\\
\bigg\langle
\left[\begin{matrix}7&2\\2&1\end{matrix}\right]
\bigg\rangle & \textrm{if}~d_K\equiv13\Mod{16}
\end{array}\right.\quad
(\simeq\mathbb{Z}/2\mathbb{Z}\times \mathbb{Z}/4\mathbb{Z}).
\end{equation*}
Decomposing
$\left[\begin{matrix}5&4\\4&1\end{matrix}\right]=
\left[\begin{matrix}5&12\\12&29\end{matrix}\right]
\left[\begin{matrix}1&0\\0&5\end{matrix}\right]$, we
deduce that
\begin{eqnarray*}
(\zeta_8^{-1}\alpha^3)^{\left[\begin{smallmatrix}5&4\\4&1\end{smallmatrix}\right]}
&=&(\zeta_8^{-1}g_{\left[\begin{smallmatrix}0\\1/2\end{smallmatrix}\right]}(\tau)^3)^{
\left[\begin{smallmatrix}5&12\\12&29\end{smallmatrix}\right]
\left[\begin{smallmatrix}1&0\\0&5\end{smallmatrix}\right]}(\tau_K)
\quad\textrm{by Proposition \ref{Gee}}\\
&=& (\zeta_8^{-1}g_{\left[\begin{smallmatrix}0\\1/2\end{smallmatrix}\right]}(\tau)^3\circ
\left[\begin{smallmatrix}5&12\\12&29\end{smallmatrix}\right])^{\left[\begin{smallmatrix}1&0\\0&5\end{smallmatrix}\right]}(\tau_K)\\
&=&(\zeta_8^{-1}g_{\left[\begin{smallmatrix}6\\29/2\end{smallmatrix}\right]}(\tau)^3)^{\left[\begin{smallmatrix}1&0\\0&5\end{smallmatrix}\right]}(\tau_K)
\quad\textrm{by Proposition \ref{transform} (ii)}\\
&=&(\zeta_8^3g_{\left[\begin{smallmatrix}0\\1/2\end{smallmatrix}\right]}(\tau)^3)
^{\left[\begin{smallmatrix}1&0\\0&5\end{smallmatrix}\right]}(\tau_K)
\quad\textrm{by Proposition \ref{transform} (iii)}\\
&=&(8\zeta_8q^{1/4}\prod_{n=1}^\infty(1+q^n)^6)
^{\left[\begin{smallmatrix}1&0\\0&5\end{smallmatrix}\right]}(\tau_K)
\quad\textrm{by (\ref{Siegel2})}\\
&=& (8\zeta_8^5q^{1/4}\prod_{n=1}^\infty(1+q^n)^6)(\tau_K)\\
&=&\zeta_8^{-1}g_{\left[\begin{smallmatrix}0\\1/2\end{smallmatrix}\right]}(\tau_K)^3\\
&=&\zeta_8^{-1}\alpha^3.
\end{eqnarray*}
In a similar way, one can verify that $\zeta_8^{-1}\alpha^3$ is
invariant under the actions of
\begin{equation*}
\left[
\begin{matrix}7&6\\2&1\end{matrix}\right]=
\left[
\begin{matrix}7&2\\10&3\end{matrix}\right]
\left[\begin{matrix}1&0\\0&3\end{matrix}\right]~\textrm{and}~
\left[
\begin{matrix}7&2\\2&1\end{matrix}\right]=
\left[
\begin{matrix}23&14\\18&11\end{matrix}\right]
\left[\begin{matrix}1&0\\0&3\end{matrix}\right].
\end{equation*}
Thus,
$\zeta_8^{-1}\alpha^3$ lies in $K_{(2)}$, so does
$\alpha^4/\zeta_8^{-1}\alpha^3=\zeta_8\alpha$. Lastly, since
$\alpha^{12}$ generates $K_{(2)}$ over $K$, so does $\zeta_8\alpha$.
Therefore, we are done.
\\(ii) Let $\alpha=g_{\left[\begin{smallmatrix}0\\1/2\end{smallmatrix}\right]}(\tau_K)$.
Since $\gamma_2(\tau_K)=(\alpha^{12}+16)/\alpha^4$ by Lemma
\ref{jtog} (ii) and $\alpha^4\in K_{(2)}\cap\mathbb{R}$ by (i), we
have $\gamma_2(\tau_K)\in K_{(2)}\cap\mathbb{R}$. Note from Remark
\ref{inertconj} that $g_{\left[\begin{smallmatrix}1/2\\0\end{smallmatrix}\right]}(\tau_K)^{12}$ and
$g_{\left[\begin{smallmatrix}1/2\\1/2\end{smallmatrix}\right]}(\tau_K)^{12}$ are the two conjugates of
$\alpha^{12}$ over $H_K$. In particular, $g_{\left[\begin{smallmatrix}1/2\\0\end{smallmatrix}\right]}(\tau_K)^{4}$
and $g_{\left[\begin{smallmatrix}1/2\\1/2\end{smallmatrix}\right]}(\tau_K)^{4}$ belong to $K_{(2)}$ by Lemma
\ref{jtog} (ii). Hence, the other two conjugates of $\alpha^4$ over
$H_K$ are $\xi_1g_{\left[\begin{smallmatrix}1/2\\0\end{smallmatrix}\right]}(\tau_K)^4$ and
$\xi_2g_{\left[\begin{smallmatrix}1/2\\1/2\end{smallmatrix}\right]}(\tau_K)^4$ for some cube roots of unity
$\xi_1,\xi_2$. If $\zeta_3$ lies in $K_{(2)}$, then $3$ ramifies in
$K_{(2)}$ (but not in $K$ by hypothesis), which contradicts the fact
that all prime ideals of $K$ which are ramified in $K_{(2)}$ must
divide $(2)$. So we get $\xi_1=\xi_2=1$. And, Lemma \ref{jtog} (ii)
shows that $\gamma_2(\tau_K)$ is invariant under the action of
$\mathrm{Gal}(K_{(2)}/H_K)$; and hence $\gamma_2(\tau_K)\in H_K$.
Therefore, we conclude that $\gamma_2(\tau_K)$ is a real algebraic
integer which generates $H_K$ over $K$ by the fact
$j(\tau_K)=\gamma_2(\tau_K)^3$ and Proposition \ref{j2} (i).
\end{proof}

\begin{remark}
Observe that
$\zeta_8g_{\left[\begin{smallmatrix}0\\1/2\end{smallmatrix}\right]}(\tau_K)=\zeta_8^3\mathfrak{f}_2(\tau_K)^2$
by (\ref{ftog}). Besides Theorem \ref{generator} there are several
other theorems which assert that the singular values of the Weber
functions and $\gamma_2(\tau)$ generate class fields of imaginary
quadratic fields (\cite[$\S$126--127]{Weber}, \cite{Schertz2},
\cite{Schertz3}, \cite[$\S$12]{Cox}) whose proofs are quite
classical. However, they are certainly elegant and worthy of
considering. On the other hand, Gee \cite{Gee} applied Shimura's
reciprocity law to the Weber functions which satisfy the following
transformation properties:
\begin{equation*}
\begin{array}{lll}
\mathfrak{f}(\tau)\circ T=\zeta_{48}^{-1}\mathfrak{f}_1(\tau), &
\mathfrak{f}_1(\tau)\circ T=\zeta_{48}^{-1}\mathfrak{f}(\tau), &
\mathfrak{f}_2(\tau)\circ T=\zeta_{24}\mathfrak{f}_2(\tau),\vspace{0.1cm}\\
\mathfrak{f}(\tau)\circ S=\mathfrak{f}(\tau), &
\mathfrak{f}_1(\tau)\circ S=\mathfrak{f}_2(\tau), &
\mathfrak{f}_2(\tau)\circ S=\mathfrak{f}_1(\tau),
\end{array}
\end{equation*}
where $T=\left[\begin{matrix}1&1\\0&1\end{matrix}\right]$
and $S=\left[\begin{matrix} 0&-1\\1&0\end{matrix}\right]$
are the generators of $\mathrm{SL}_2(\mathbb{Z})$. But, the general
transformation formula for
\begin{equation*}
f(\tau)\circ\gamma\quad(f(\tau)=\mathfrak{f}(\tau),\mathfrak{f}_1(\tau),
\mathfrak{f}_2(\tau),~ \gamma\in\mathrm{SL}_2(\mathbb{Z}))
\end{equation*}
does not seem to be known, which forces her to produce a redundant
step to decompose $\gamma$ into a product of $T$ and $S$
\cite[$\S$5]{Gee}. Thus we would like to point out that the
relation (\ref{ftog}) and Proposition \ref{transform} (ii), (iii)
will give us an explicit formula for $f(\tau)^2\circ\gamma$, from
which one can efficiently apply Shimura's reciprocity law.
\end{remark}

\section{Application to class number one problem}\label{sec5}

In this section we shall revisit Gauss' class number one problem for
imaginary quadratic fields.
\par
Let $K$ be an imaginary quadratic field of discriminant $d_K$. Since
$j(\tau_K)$ is a real algebraic integer lying in $H_K$ by
Proposition \ref{j2} (i), it should be an integer when $K$ has class
number one. By determining the form class group $\mathrm{C}(d_K)$ we
know that there are only nine imaginary quadratic fields $K$ of
class number one with $d_K\geq-163$ \cite[p.261]{Cox}:

\begin{table}[h]
{\footnotesize
\begin{tabular} {c|ccccccccc}\hline $K$ &
$\mathbb{Q}(\sqrt{-3})$ & $\mathbb{Q}(\sqrt{-1})$ &
$\mathbb{Q}(\sqrt{-7})$ & $\mathbb{Q}(\sqrt{-2})$ &
$\mathbb{Q}(\sqrt{-11})$ & $\mathbb{Q}(\sqrt{-19})$ &
$\mathbb{Q}(\sqrt{-43})$ &
$\mathbb{Q}(\sqrt{-67})$ & $\mathbb{Q}(\sqrt{-163})$ \\
$d_K$ & $-3$ & $-4$ & $-7$ & $-8$ & $-11$ & $-19$ &
$-43$ & $-67$ & $-163$\\
$j(\tau_K)$ & $0$ & $12^3$ & $-15^3$ & $20^3$ & $-32^3$ &
$-96^3$ & $-960^3$ & $-5280^3$ & $-640320^3$\\
$\gamma_2(\tau_K)$ & $0$ & $12$ & $-15$ & $20$ & $-32$ &
$-96$ & $-960$ & $-5280$ & $-640320$\\
\hline
\end{tabular}
}
\caption{Imaginary quadratic fields $K$ of class number one with
$d_K\geq-163$.}\label{Table}
\end{table}

We shall show in this section that the above table is the complete
one by utilizing Shimura's reciprocity law and Siegel functions.

\begin{lemma}\label{inequiv}
Let $\tau_0\in\mathbb{H}$, and set $A=|e^{2\pi i\tau_0}|$.
\begin{itemize}
\item[\textup{(i)}] If $\left[\begin{matrix}a\\b\end{matrix}\right]\in\mathbb{Q}^2$ with $0<a\leq1/2$,
then $|g_{\left[\begin{smallmatrix}a\\b\end{smallmatrix}\right]}(\tau_0)|\leq
A^{(1/2)\mathbf{B}_2(a)}e^{2A^a/(1-A)}$.
\item[\textup{(ii)}] If $b\in\mathbb{Q}$ with $0<b<1$, then
$|g_{\left[\begin{smallmatrix}0\\b\end{smallmatrix}\right]}(\tau_0)|\leq A^{(1/2)\mathbf{B}_2(0)}|1-e^{2\pi
ib}|e^{2A/(1-A)}$.
\end{itemize}
\end{lemma}
\begin{proof}
(i) We derive from the definition (\ref{Siegel}) that
\begin{eqnarray*}
|g_{\left[\begin{smallmatrix}a\\b\end{smallmatrix}\right]}(\tau_0)|&\leq& A^{(1/2)\mathbf{B}_2(a)}(1+A^a)
\prod_{n=1}^\infty(1+A^{n+a})(1+A^{n-a})\\
&\leq&A^{(1/2)\mathbf{B}_2(a)}\prod_{n=0}^\infty(1+A^{n+a})^2\quad\textrm{by
the facts
$A<1$ and $0<a\leq1/2$}\\
&\leq&A^{(1/2)\mathbf{B}_2(a)}\prod_{n=0}^\infty e^{2A^{n+a}}\quad
\textrm{by the inequality $1+X<e^X$ for $X>0$}\\
&=&A^{(1/2)\mathbf{B}_2(a)}e^{2A^a/(1-A)}.
\end{eqnarray*}
(ii) In a similar way, we get that
\begin{eqnarray*}
|g_{\left[\begin{smallmatrix}0\\b\end{smallmatrix}\right]}(\tau_0)|&\leq&A^{(1/2)\mathbf{B}_2(0)}|1-e^{2\pi
ib}|\prod_{n=1}^\infty(1+A^n)^2\\
&\leq&A^{(1/2)\mathbf{B}_2(0)}|1-e^{2\pi ib}|\prod_{n=1}^\infty
e^{2A^n}\quad
\textrm{by the inequality $1+X<e^X$ for $X>0$}\\
&=&A^{(1/2)\mathbf{B}_2(0)}|1-e^{2\pi ib}|e^{2A/(1-A)}.
\end{eqnarray*}
\end{proof}

\begin{theorem}\label{classnumberone}
Let $K$
\textup{(}$\neq\mathbb{Q}(\sqrt{-3}),\mathbb{Q}(\sqrt{-1})$\textup{)}
be an imaginary quadratic field of discriminant $d_K$. Assume that
$K$ has class number one \textup{(}that is, $H_K=K$\textup{)}.
\begin{itemize}
\item[\textup{(i)}] If $2$ is not inert in $K$, then $d_K=-7,-8$.
\item[\textup{(ii)}] If $2$ is inert and $3$ is ramified in $K$, then
there is no such $K$.
\item[\textup{(iii)}] If $2$ is inert and $3$ is not ramified in $K$, then
$d_K=-11,-19,-43,-67,-163$.
\end{itemize}
\end{theorem}
\begin{proof}
Since we are assuming that $K$ is neither $\mathbb{Q}(\sqrt{-3})$
nor $\mathbb{Q}(\sqrt{-1})$, we have $d_K\leq -7$. Let
$\alpha=g_{\left[\begin{smallmatrix}0\\1/2\end{smallmatrix}\right]}(\tau_K)$ ($\neq0$) and
$A=e^{-\pi\sqrt{|d_K|}}$.\\
(i) If $d_K\leq-31$, then we see that
\begin{eqnarray*}
&&|\mathbf{N}_{K_{(2)}/K}(\alpha^{12})|^{1/12}\\
&=& \left\{
\begin{array}{ll}
|g_{\left[\begin{smallmatrix}0\\1/2\end{smallmatrix}\right]}(\tau_K)
g_{\left[\begin{smallmatrix}1/2\\1/2\end{smallmatrix}\right]}(\tau_K)| &
\textrm{if}~d_K\equiv0\Mod{8},
\vspace{0.1cm}\\
|g_{\left[\begin{smallmatrix}0\\1/2\end{smallmatrix}\right]}(\tau_K)
g_{\left[\begin{smallmatrix}1/2\\0\end{smallmatrix}\right]}(\tau_K)| &
\textrm{if}~d_K\equiv4\Mod{8},\vspace{0.1cm}\\
|g_{\left[\begin{smallmatrix}0\\1/2\end{smallmatrix}\right]}(\tau_K)| & \textrm{if}~d_K\equiv1\Mod{8}
\end{array}\right.\quad\textrm{by (\ref{norm})}\\
&\leq&\left\{
\begin{array}{ll}
2A^{1/12}e^{2A/(1-A)}\cdot A^{-1/24}e^{2A^{1/2}/(1-A)} &
\textrm{if}~d_K\equiv0\Mod{4},
\vspace{0.1cm}\\
2A^{1/12}e^{2A/(1-A)} & \textrm{if}~d_K\equiv1\Mod{8}
\end{array}
\right.\quad\textrm{by Lemma
\ref{inequiv}}\\
&<&1\quad\textrm{by the fact $A\leq e^{-\pi\sqrt{31}}$}.
\end{eqnarray*}
On the other hand, since $\mathbf{N}_{K_{(2)}/K}(\alpha^{12})$ is a
nonzero integer by Theorem \ref{properties} (ii), the above
inequality is false; hence $d_K>-31$. And, we get the conclusion by
Table \ref{Table}.\\
(ii) Since $2$ is inert and $3$ is ramified in $K$ (equivalently,
$d_K\equiv21\Mod{24}$), we have
\begin{equation*}
\mathrm{Gal}(K_{(3)}/K)\simeq
W_{3,\tau_K}/\mathrm{Ker}_{3,\tau_K}= \bigg\{
\left[\begin{matrix}1&0\\0&1\end{matrix}\right],
\left[\begin{matrix}1&0\\1&1\end{matrix}\right],
\left[\begin{matrix}2&0\\1&2\end{matrix}\right]
\bigg\}
\end{equation*}
by Proposition \ref{Gee}. Let $\beta=g_{\left[\begin{smallmatrix}0\\1/3\end{smallmatrix}\right]}(\tau_K)$
($\neq0$). Since $g_{\left[\begin{smallmatrix}0\\1/3\end{smallmatrix}\right]}(\tau)^{12}\in\mathcal{F}_3$ by
Proposition \ref{level}, we have $\beta^{12}\in K_{(3)}$ by
Proposition \ref{CM}. Furthermore, since $\beta^{12}$ is a real
algebraic integer by the definition (\ref{Siegel}), Propositions
\ref{transform} (vi) and \ref{j2} (i),
$\mathbf{N}_{K_{(3)}/K}(\beta^{12})$ is a nonzero integer by Lemma
\ref{intcoeff}. If $d_K\leq-51$, then we derive that
\begin{eqnarray*}
|\mathbf{N}_{K_{(3)}/K}(\beta^{12})|^{1/12}
&=&|g_{\left[\begin{smallmatrix}0\\1/3\end{smallmatrix}\right]}(\tau_K) g_{\left[\begin{smallmatrix}1/3\\1/3\end{smallmatrix}\right]}(\tau_K)
g_{\left[\begin{smallmatrix}1/3\\2/3\end{smallmatrix}\right]}(\tau_K)|\quad\textrm{by Proposition \ref{transform} (iv) and (v)}\\
&\leq&A^{1/12}|1-\zeta_3| e^{2A/(1-A)}
\cdot(A^{-1/36}e^{2A^{1/3}/(1-A)})^2\quad
\textrm{by Lemma \ref{inequiv}}\\
&=&\sqrt{3}A^{1/36}e^{(2A+4A^{1/3})/(1-A)}\\
&<&1 \quad\textrm{by the fact $A\leq e^{-\pi\sqrt{51}}$}.
\end{eqnarray*}
Hence we must have $-51<d_K\leq-7$. But, there is no such imaginary
quadratic field $K$ with $d_K\equiv21\Mod{24}$ as desired.\\
(iii) Let $x=\zeta_8\alpha$. Since $[K_{(2)}:K]=3$ by the degree
formula (\ref{deg}), we have
\begin{equation*}
\min(x,K)=X^3+aX^2+bX+c\quad\textrm{for some}~a,b,c\in\mathbb{Z}
\end{equation*}
by Theorem \ref{generator} (i) and Lemma \ref{intcoeff}. Furthermore,
we get
\begin{equation*}
\min(x^4,K)=X^3-\gamma_2(\tau_K)X-16\quad(\in\mathbb{Z}[X])
\end{equation*}
by Lemma \ref{jtog} (ii) and Theorem \ref{generator} (ii). Now, by
adopting Heegner's idea \cite{Heegner} one can determine the
possible values of $a,b,c$, from which we obtain
\begin{equation*}
\gamma_2(\tau_K)=0,-32,-96,-960,-5280,-640320.
\end{equation*}
Therefore, we can conclude the assertion (iii) by Table \ref{Table}
and Lemma \ref{j}, although we omit the details
\cite[pp.272--274]{Cox}.
\end{proof}

\begin{remark}
\begin{itemize}
\item[(i)] In 1903 Landau (\cite{Landau} or \cite[Theorem 2.18]{Cox}) presented a simple and elementary proof of
Theorem \ref{classnumberone} (i) by considering the form class group
$\mathrm{C}(d_K)$.
\item[(ii)] Theorem \ref{generator} (i) is essentially a gap in
Heegner's work, which was fulfilled by Stark \cite{Stark}.
\item[(iii)] To every imaginary quadratic order $\mathcal{O}$ of
class number one there is an associated elliptic curve
$E_\mathcal{O}$ over $\overline{\mathbb{Q}}$ admitting complex
multiplication by $\mathcal{O}$. It can be defined over $\mathbb{Q}$
and is unique up to $\overline{\mathbb{Q}}$-isomorphism. For a
positive integer $n$, let $X_{\textrm{ns}}^+(n)$ be the modular
curve associated to the normalizer of the non-split Cartan subgroup
of level $n$ which can be defined over $\mathbb{Q}$ \cite{Baran2}.
If every prime $p$ dividing $n$ is inert in $\mathcal{O}$, then
$E_\mathcal{O}$ gives rise to an integral point of
$X_{\textrm{ns}}^+(n)$ \cite[p.195]{Serre}. Here, by integral
points we mean the points corresponding to elliptic curves with
integral $j$-invariant. As Serre pointed out \cite[p.197]{Serre},
the solutions by Heegner and Stark can be viewed as the
determination of the integral points of $X_\mathrm{ns}^+(24)$. And,
Baran \cite{Baran} recently gave a geometric solution of the class
number one problem by finding an explicit parametrization for the
modular curve $X_{\textrm{ns}}^+(9)$ over $\mathbb{Q}$.
\end{itemize}
\end{remark}

We can also apply the arguments in the proof of Theorem
\ref{classnumberone} (i) and (ii) to solve a problem concerning imaginary
quadratic fields of class number two.

\begin{theorem}
$\mathbb{Q}(\sqrt{-15})$ is the unique imaginary quadratic field of
class number two in which $2$ splits.
\end{theorem}
\begin{proof}
Let $K$ be an imaginary quadratic field of discriminant $d_K$ and
class number two in which $2$ splits (so $d_K\equiv1\Mod{8}$). Then
the form class group $\mathrm{C}(d_K)$ consists of two reduced
quadratic forms, that is
\begin{eqnarray*}
&&Q_1=X^2+XY+((1-d_K)/4)Y^2,\\
&&Q_2=aX^2+bXY+cY^2\quad\textrm{for some $a,b,c\in\mathbb{Z}$ with
$2\leq a\leq\sqrt{|d_K|/3}$}.
\end{eqnarray*}
And, we have by Proposition \ref{Hilbert} with $p=2$
\begin{eqnarray*}
&&\tau_{Q_1}=(-1+\sqrt{d_K})/2,\quad u_{Q_1}=\left[\begin{matrix}
1&0\\0&1
\end{matrix}\right],\\
&&\tau_{Q_2}=(-b+\sqrt{d_K})/2a,\quad u_{Q_2}=\left[\begin{matrix}
*&*\\r&s
\end{matrix}\right]~\textrm{for some}~\left[\begin{matrix}r\\s\end{matrix}\right]\in\bigg\{
\left[\begin{matrix}0\\1\end{matrix}\right],\left[\begin{matrix}1\\0\end{matrix}\right],
\left[\begin{matrix}1\\1\end{matrix}\right]\bigg\}.
\end{eqnarray*}
Let $\alpha=g_{\left[\begin{smallmatrix}0\\1/2\end{smallmatrix}\right]}(\tau_K)$ ($\neq0$). Then $\alpha^{12}\in
H_K$ by (\ref{norm}). If $d_K\leq-31$, then we derive that
\begin{eqnarray*}
&&|\mathbf{N}_{H_K/K}(\alpha^{12})|^{1/12}\\
&=&\prod_{Q\in\mathrm{C}(d_K)}|(g_{\left[\begin{smallmatrix}0\\1/2\end{smallmatrix}\right]}(\tau)^{12})^{u_Q}(\tau_Q)|^{1/12}
\quad\textrm{by Proposition \ref{Hilbert}}\\
&=&|g_{\left[\begin{smallmatrix}0\\1/2\end{smallmatrix}\right]}((-1+\sqrt{d_K})/2)|\times
|g_{\left[\begin{smallmatrix}r/2\\s/2\end{smallmatrix}\right]}((-b+\sqrt{d_K})/2a)|\quad\textrm{by
Proposition \ref{transform} (iv) and (v)}\\
&\leq&2A^{1/12}e^{2A/(1-A)}\times \left\{
\begin{array}{ll}
2A^{1/12a}e^{2A^{1/a}/(1-A^{1/a})} & \textrm{if}~r=0,
\vspace{0.1cm}\\
A^{-1/24a}e^{2A^{1/2a}/(1-A^{1/a})}
 & \textrm{if}~r=1
\end{array}\right.\quad\textrm{with $A=e^{-\pi\sqrt{|d_K|}}$ by Lemma \ref{inequiv}}\\
&\leq& \left\{\begin{array}{ll} 4A^{1/12}
e^{2A/(1-A)-\pi\sqrt{3}/12+2e^{-\pi\sqrt{3}}/(1-e^{-\pi\sqrt{3}})} &
\textrm{if}~r=0,
\vspace{0.1cm}\\
2A^{1/12-1/48} e^{2A/(1-A)+
2e^{-\pi\sqrt{3}/2}/(1-e^{-\pi\sqrt{3}})} & \textrm{if}~r=1
\end{array}\right.\quad\textrm{because $A<1$ and $2\leq
a\leq\sqrt{|d_K|/3}$}\\
&<&1\quad\textrm{by the fact $A\leq e^{-\pi\sqrt{31}}$}.
\end{eqnarray*}
On the other hand, since $\alpha^{12}$ is a real algebraic integer,
$\mathbf{N}_{H_K/K}(\alpha^{12})$ is a nonzero integer by Lemma
\ref{intcoeff}. Therefore we should have $d_K>-31$. One can then
easily see by the following remark that $\mathbb{Q}(\sqrt{-15})$ is
the unique one.
\end{proof}

\begin{remark}
There are exactly eighteen imaginary quadratic fields of class
number two whose discriminants are as follows
\cite[p.636]{Ribenboim}:
\begin{equation*}
-15,-20,-24,-35,-40,-51,-52,-88,-91,-115,
-123,-148,-187,-232,-235,-267,-403,-427.
\end{equation*}
\end{remark}

\bibliographystyle{amsplain}

\begin{thebibliography}{10}

\bibitem {Baker} A. Baker, \textit{Linear forms in the logarithms of
algebraic numbers I}, Mathematika 13 (1966), 204--216.

\bibitem {Baran} B. Baran,
\textit{A modular curve of level $9$ and the class number one
problem}, J. Number Theory 129 (2009), no. 3, 715--728.

\bibitem {Baran2} B. Baran, \textit{Normalizers of non-split Cartan subgroups, modular curves, and
the class number one problem}, J. Number Theory 130 (2010), no. 12,
2753--2772.

\bibitem {Cox} D. A. Cox, \textit{Primes of the form $x^2+ny^2$: Fermat, Class Field, and Complex Multiplication},
A Wiley-Interscience Publication, John Wiley \& Sons, Inc., New
York, 1989.

\bibitem {Deuring} M. Deuring, \textit{Imagin\"{a}re
quadratische Zahlk\"{o}rper mit der Klassenzahl Eins}, Invent. Math.
5 (1968), 169--179.

\bibitem {Dorman} D. R. Dorman, \textit{Singular moduli, modular polynomials, and the
index of the closure of $\mathbb{Z}[j(\tau)]$ in
$\mathbb{Q}(j(\tau))$}, Math. Ann. 283 (1989), no. 2, 177--191.

\bibitem {Gee} A. Gee, \textit{Class invariants by Shimura's reciprocity law},
J. Th\'{e}or. Nombres Bordeaux 11 (1999), no. 1, 45--72.

\bibitem{G-Z} B. H. Gross and D. B. Zagier, \textit{On singular
moduli}, J. Reine Angew. Math. 355 (1985), 191--220.

\bibitem {Heegner} K. Heegner, \textit{Diophantische Analysis und
Modulfunktionen}, Math. Zeit. 56 (1952), 227--253.

\bibitem {Janusz} G. J. Janusz, \textit{Algebraic Number Fields},
2nd edition, Grad. Studies in Math. 7, Amer. Math. Soc., Providence,
R. I., 1996.

\bibitem {K-S} J. K. Koo and D. H. Shin, \textit{On some arithmetic properties of Siegel functions},
Math. Zeit. 264 (2010), no. 1, 137--177.

\bibitem {K-L} D. Kubert and S. Lang, \textit{Modular Units}, Grundlehren der mathematischen Wissenschaften 244, Spinger-Verlag, New York-Berlin, 1981.

\bibitem {Landau} E. Landau, \textit{\"{U}ber die Klassenzahl der binaren quadratischen
Formen von negativer Discriminante}, Math. Ann. 56 (1903), no. 4,
671--676.

\bibitem {Lang3} S. Lang, \textit{Introduction to Modular Forms},
Grundlehren der mathematischen Wissenschaften, No. 222,
Springer-Verlag, Berlin-New York, 1976.

\bibitem {Lang} S. Lang, \textit{Elliptic Functions},
2nd edn, Grad. Texts in Math. 112,
Spinger-Verlag, New York, 1987.

\bibitem {Lang2} S. Lang, \textit{Algebraic Number Theory}, 2nd edn, Spinger-Verlag,
New York, 1994.

\bibitem {Ramachandra}
K. Ramachandra, \textit{Some applications of Kronecker's limit
formula}, Ann. of Math. (2) 80 (1964), 104--148.

\bibitem {Ribenboim} P. Ribenboim, \textit{Classical Theory of Algebraic
Numbers}, Universitext, Springer-Verlag, New York, 2001.

\bibitem {Schertz2} R. Schertz,
\textit{Die singul\"{a}ren Werte der Weberschen Funktionen
$\mathfrak{f}$, $\mathfrak{f}_1$, $\mathfrak{f}_2$, $\gamma_2$,
$\gamma_2$}, J. Reine Angew. Math. 286/287 (1976), 46--74.

\bibitem {Schertz} R. Schertz, \textit{Construction of ray class fields by elliptic
units}, J. Th\'{e}or. Nombres Bordeaux 9 (1997), no. 2, 383--394.

\bibitem {Schertz3} R. Schertz, \textit{Weber's class invariants revisited},
J. Th\'{e}or. Nombres Bordeaux 14 (2002), no. 1, 325--343.

\bibitem {Serre} J-P. Serre, \textit{Lectures on the Mordell-Weil
Theorem}, Aspects of Mathematics, E15, Vieweg $\&$ Sohn,
Braunschweig, 1989.

\bibitem {Shimura} G. Shimura, \textit{Introduction to the Arithmetic Theory of Automorphic Functions},
Iwanami Shoten and Princeton University Press, Princeton, N. J.,
1971.

\bibitem {Stark2} H. M. Stark, \textit{A complete determination of
the complex quadratic fields of class number one}, Michigan Math. J.
14 (1967), 1--27.

\bibitem {Stark} H. M. Stark, \textit{On the ``gap" in a theorem of
Heegner}, J. Number Theory 1 (1969), 16--27.


\bibitem {Stevenhagen} P. Stevenhagen, \textit{Hilbert's 12th problem, complex multiplication and Shimura
reciprocity}, Class Field Theory-Its Centenary and Prospect (Tokyo,
1998), 161--176, Adv. Stud. Pure Math., 30, Math. Soc. Japan, Tokyo,
2001.

\bibitem {Weber} H. Weber, \textit{Lehrbuch der Algebra}, Vol. III.
2nd edn, Vieweg, Braunschwieg, 1908 (Reprint by Chelsea, New
York, 1961).

\end{thebibliography}

\address{
Department of Mathematics\\
Hankuk University of Foreign Studies\\
Yongin-si, Gyeonggi-do 449-791\\
Republic of Korea } {dhshin@hufs.ac.kr}

\end{document}